\documentclass[11pt, a4paper]{article}

\usepackage{amssymb,amsmath,amsthm}
\usepackage{fancyhdr}
\usepackage[british]{babel}
\usepackage{enumitem}

\usepackage[usenames,dvipsnames,table]{xcolor}
\usepackage{tikz}
\usetikzlibrary{decorations.pathreplacing}
\usetikzlibrary{decorations.pathmorphing}
\usepackage{bbm}

\usepackage{kotex}
\usepackage{todonotes}
\usepackage{subcaption}

\usepackage{hyperref}
\usepackage[margin=2.5cm]{geometry}
\usepackage{thmtools,thm-restate}

\newtheorem{definition}{Definition}[section]
\newtheorem{theorem}[definition]{Theorem}

\newtheorem{lemma}[definition]{Lemma}
\newtheorem{claim}{Claim}[lemma]
\newtheorem{claim*}{Claim}[lemma]
\newtheorem{claim**}{Claim}[lemma]

\numberwithin{equation}{section}

\newcommand{\nin}{\notin}

\newcommand{\ignore}[1]{}

\title{ On Ideal Secret-Sharing Schemes for $k$-homogeneous access structures }

\author{Younjin Kim
\thanks{Extremal Combinatorics and Probability Group (ECOPRO), Institute for Basic Science (IBS), Daejeon, South Korea. Email: {\tt younjinkim@ibs.re.kr}.
Y.K. was supported by the Institute for Basic Science (IBS-R029-C4).}\and Jihye Kwon
 \thanks{Department of Mathematics, Ewha Womans University, Seoul, South Korea. Email: {\tt jhkwon74@ewhain.net}. J.K. was supported by the National Research Foundation of Korea(NRF) grant funded by the Ministry of Education (No. 2019R1A6A1A11051177)}.\and
 Hyang-Sook Lee 
 \thanks{Department of Mathematics, Ewha Womans University, Seoul, South Korea. Email: {\tt hsl@ewha.ac.kr}. H.L. was supported by the National Research Foundation of Korea(NRF) grant funded by the Ministry of Education (No. 2019R1A6A1A11051177) and partially by the Korea Government(MSIT) (No. NRF-2021R1A2C1094821).}
 }

\begin{document}

\maketitle

\begin{abstract}

A $k$-uniform hypergraph is a hypergraph where each $k$-hyperedge has exactly $k$ vertices. A $k$-homogeneous access structure is represented by a $k$-uniform hypergraph $\mathcal{H}$, in which the participants correspond to the vertices of hypergraph $\mathcal{H}$. A set of vertices can reconstruct the secret value from their shares if they are connected by a $k$-hyperedge, while a set of non-adjacent vertices does not obtain any information about the secret. One parameter for measuring the efficiency of a secret sharing scheme is the information rate, defined as the ratio between the length of the secret and the maximum length of the shares given to the participants. Secret sharing schemes with an information rate equal to one are called ideal secret sharing schemes. An access structure is considered ideal if an ideal secret sharing scheme can realize it. Characterizing ideal access structures is one of the important problems in secret sharing schemes. The characterization of ideal access structures has been studied by many authors~\cite{BD, CT,JZB, FP1,FP2,DS1,TD}.  In this paper, we characterize ideal $k$-homogeneous access structures using the independent sequence method. In particular, we prove that the reduced access structure of $\Gamma$ is an $(k, n)$-threshold access structure when the optimal information rate of $\Gamma$ is larger than $\frac{k-1}{k}$, where $\Gamma$ is a $k$-homogeneous access structure satisfying specific criteria.
  \end{abstract}

\section{Introduction}

A secret sharing scheme is a tool utilized in numerous cryptographic protocols. It involves a dealer who possesses a secret, a set of $n$ participants, and a collection $\mathcal{F}$ of subsets of participants defined as the access structure. A secret sharing scheme for $\mathcal{F}$ is a method in which the dealer distributes shares of a secret value $k$ among the $n$ participants so that any subset within $\mathcal{F}$ can reconstruct the secret value $k$ from their shares, while any subset not in $\mathcal{F}$ cannot reveal any information about the secret value $k$. The {\it qualified subsets} in the secret sharing scheme are defined as the subsets of participants capable of reconstructing the secret value $k$ from their shares. A collection of qualified subsets of participants is referred to as the {\it access structure} of the secret sharing scheme. In contrast, the {\it unqualified subsets} or {\it forbidden subsets} in the secret sharing scheme are defined as the subsets of participants who cannot obtain any information about the secret value $k$ from their shares.\\

In 1979, Shamir~\cite{AS} introduced a $(t,n)$-threshold secret sharing scheme as the pioneering work in secret sharing. In this scheme, the qualified subsets consist of all subsets with at least $t$ participants from a set of $n$ participants, and the size of each share is equal to the size of the secret. This implies that the $(t,n)$-threshold secret sharing scheme is determined by the basis consisting of all subsets with exactly $t$ distinct participants from a set of $n$ participants.\\

A hypergraph is a generalization of a graph in which hyperedges may connect more than two vertices. A $k$-uniform hypergraph is one in which each hyperedge has exactly $k$ vertices.
An access structure is a $k$-uniform hypergraph access structure, denoted by $\mathcal{H}$, if the set of vertices connected by a $k$-hyperedge can reconstruct the secret, and the set of non-adjacent vertices in $\mathcal{H}$ does not reveal any information about the secret. Access structures of this type are also called $k$-homogeneous.
A $k$-homogeneous access structure is determined by a family of minimal qualified subsets, each consisting of exactly $k$ different participants, or $k$-uniform hypergraphs $\mathcal{H}(V,E)$, where $V$ is a vertex set and $E \subseteq 2^V$ is the edge set of hyperedges of cardinality $k$.
Several authors have constructed secret sharing schemes for $k$-homogeneous access structures using various techniques.\\

One way to measure the efficiency of a secret sharing scheme is to use the information rate, defined as the ratio between the length of the secret and the maximum length of the shares given to the participants. Since the length of any share is greater than or equal to the length of the secret in a secret sharing scheme, the information rate cannot exceed one. Secret sharing schemes with an information rate equal to one are called {\it ideal secret sharing schemes}. For example, Shamir's threshold secret sharing scheme is ideal.\\

 An access structure is ideal if there exists an ideal secret sharing scheme to implement it. The characterization of ideal access structures is an important issue in secret sharing schemes, which has been studied by numerous authors~\cite{BD,CT,JZB,FP1,FP2,MP3,DS1,TD}.
 In 1992, Stinson~\cite{DS1} exactly characterized all ideal $2$-homogeneous access structures. In 2007, Mart\'i-Farr\'e and Padr\'o~\cite{MP3} characterized all ideal $3$-homogeneous access structures in which the number of minimal qualified subsets contained in any set of four participants is not equal to three. Later, in 2009, Mart\'i-Farr\'e and Padr\'o~\cite{FP2}  also characterized ideal rank-three access structures in some cases. Also, in 2009, Tassa and Dyn~\cite{TD} studied an ideal secret sharing scheme that realizes compartmented access structures using bivariate interpolation. Recently, in 2021, Janbaz and Bagherpour~\cite{JZB} characterized ideal graph-based $3$-homogeneous access structures. In this paper, we characterize the ideal $k$-homogeneous access structures by utilizing the independent sequence method.

\begin{theorem}\label{main:mainthm}
Suppose $\Gamma$ is a $k$-homogeneous access structure on a set of participants $\mathcal{P}$ such that the number of minimal qualified subsets contained in any set of $k+1$ participants is not equal to $1$ and $k$, containing  $k+1$. Then the following conditions are equivalent.

\begin{itemize}
\item[(1)] $\Gamma$ is a vector space access structure
\item[(2)] $\Gamma$ is an ideal access structure
\item[(3)] $\rho^* (\Gamma) > \frac{k-1}{k}$
\item[(4)] The reduced access structure of $\Gamma$ is an $(k,n)$-threshold access structure.
\end{itemize}

\end{theorem}

Our paper is organized as follows. In Section 2, we introduce the definitions of an Ideal Secret Sharing Scheme and the Independent Sequence Method. In Section 3, we present several access structures related to this paper. In Section 4, we present the necessary results and lemmas required for proving our main theorem. Finally, in Section 5, we provide the proof of Theorem~\ref{main:mainthm}.

\section{Secret Sharing Scheme}

\noindent A secret sharing scheme consists of a dealer who possesses a secret, a set of $n$ participants, and a collection $\mathcal{F}$ of subsets of participants defined as the access structure. A secret sharing scheme for $\mathcal{F}$ is a method by which the dealer distributes shares of a secret value $k$ to the set of $n$ participants such that any qualified subset in $\mathcal{F}$ can reconstruct the secret value $k$ from their shares, while any unqualified subset not in $\mathcal{F}$ cannot reveal any information about the secret value $k$.

\subsection{Ideal Secret Sharing Schemes}

One parameter used to measure the efficiency of a secret sharing scheme is the information rate, which is defined as the ratio between the length of the secret and the maximum length of the shares given to the participants, as follows. 

\begin{definition}[information rate]

Let $\mathcal{P}$ be the set of all participants, $\mathcal{S}$ be the set of all secret keys, and $K(p)$ be the set of all possible shares given to a participant $p \in \mathcal{P}$. In the secret sharing scheme $\mathcal{F}$, the information rate, denoted by $\rho(\mathcal{F})$, is defined as 

$$ \rho(\mathcal{F}) = \frac{\log|S|}{\max_{p\in \mathcal{P}}\log|K(p)|} \ .$$
\end{definition}

\noindent Since the length of any share is greater than or equal to the length of the secret in the secret sharing scheme, the information rate cannot be greater than one. Therefore, $\rho(\mathcal{F}) =1$ is the optimal situation. When designing a secret sharing scheme for the given access structure $\Gamma$, we may try to maximize the information rate, as defined below.

\begin{definition}[optimal information rate]

In the secret sharing scheme $\mathcal{F}$, the optimal information rate of the access structure $\Gamma$, denoted by $\rho^*(\Gamma)$, is defined as

$$ \rho^*(\Gamma) = \sup ( \rho(\mathcal{F}))$$
\end{definition}

\noindent where the supremum is taken over all possible secret sharing schemes $\mathcal{F}$ with access structure $\Gamma$. Of course, the optimal information rate of an ideal access structure is equal to one.\\

Secret sharing schemes with an information rate equal to one are called {\it ideal secret sharing schemes}. An access structure is ideal if there exists an ideal secret sharing scheme that realizes it. Characterizing ideal access structures and providing bounds on the optimal information rate are two important problems in the  secret sharing schemes. These problems have been studied extensively in several particular families of access structures by many authors~\cite{BD,CT,JZB,FP1,FP2,MP3,DS1,TD}.

\subsection{Independent Sequence Methods} 

To prove the main theorem, we will utilize the independent sequence method, which we introduce in this section. In 1997, Blundo, Santis, Simone, and Vaccaro~\cite{BSSV2} presented the independent sequence method as a way to find upper bounds on the optimal information rate. Later, in 2000, Padr\'o and S\'aez~\cite{PS} introduced a slight generalization of this method, as follows.\\

\begin{definition} [Independent Sequence]Let $\Gamma$ be an access structure on a set of participants $\mathcal{P}$. A sequence of non-empty sets $B_1, B_2, \cdots, B_m$, 
 where 
$$ \emptyset \neq B_1 \subset B_2 \subset \cdots \subset B_m \subset \mathcal{P},$$

\noindent is called {\it independent} if 

\begin{itemize}
\item[1.] $B_m \not \in \Gamma$
\item[2.] there exist $X_1, X_2, \cdots, X_m \subset \mathcal{P}$ such that $B_i \cup X_i \in \Gamma$ and $B_{i-1} \cup X_i \not \in \Gamma$ for all $i=1,2, \cdots, m$, where $B_0=\emptyset$.
\end{itemize}

\end{definition}
\begin{theorem}[Independent Sequence Method~\cite{PS}]
Let $\Gamma$ be an access structure on a set of participants $\mathcal{P}$. Suppose that $\emptyset \neq B_1 \subset B_2 \subset \cdots \subset B_m \subset \mathcal{P}$ is an independent sequence, and let $ A\subset \mathcal{P}$ be a minimum set such that  $ \bigcup^{m}_{i=1} X_i \subset A$  which makes this sequence independent. Then, we have 

\begin{itemize}
\item[(i)] $\rho^*(\Gamma) \leq \frac{|A|}{m+1}$ \ \   \ if  $A \in \Gamma$ 
\item[(ii)] $\rho^*(\Gamma) \leq \frac{|A|}{m}$ \ \   \ if  $A  \not \in \Gamma$.
\end{itemize}

\end{theorem}

\section{Access Structures}

In a secret sharing scheme, the qualified subsets are the sets of participants who can reconstruct the secret value from their shares. The collection of qualified subsets of participants is called the access structure of the secret sharing scheme. Conversely, the unqualified subsets are the sets of participants who cannot obtain any information about the secret value from their shares.
In any secret sharing scheme, the access structure is said to  be monotone if the superset of any qualified subset is also a qualified subset, and it is determined by the family of minimal qualified subsets of participants. The collection of minimal qualified subsets of participants is called the basis of the access structure. Moreover, in a secret sharing scheme, every participant must belong to at least one minimal qualified subset.

\subsection{$k$-Homogeneous Access Structures}

A hypergraph is a generalization of a graph in which hyperedges may connect more than two vertices. A $k$-uniform hypergraph (or $k$-hypergraph) is a hypergraph in which each hyperedge has exactly $k$ vertices, or is a $k$-hyperedge.
In particular, the complete $k$-uniform hypergraph on $n$ vertices has all $k$-subsets of $\{1,2,\dots,n\}$ as $k$-hyperedges.
A $k$-hypergraph access structure is represented by a $k$-uniform hypergraph $\mathcal{H}$ in which the participants correspond to the vertices of $\mathcal{H}$. A set of vertices can reconstruct the secret value from their shares if they are connected by a $k$-hyperedge, while the set of non-adjacent vertices receives no information on the secret. A $k$-hypergraph access structure is also called a {\it $k$-homogeneous access structure}.\\

A $k$-homogeneous access structure is determined by the family of minimal qualified subsets, each consisting of exactly $k$ participants. Recall that in the $(k,n)$-threshold secret sharing scheme, the qualified subsets are formed by all subsets with at least $k$ participants among the set of $n$ participants. For example, consider an access structure on a set of five participants $P=\{p_1,p_2,p_3,p_4,p_5\}$ with minimal qualified subsets $A_1=\{p_1,p_2,p_3\}, A_2=\{p_2,p_3,p_4\}$, and $A_3=\{p_3,p_4,p_5\}$. This access structure is not $(3,n)$-threshold but is $3$-homogeneous. Note that there is a one-to-one correspondence between $k$-uniform hypergraphs and $k$-homogeneous access structures. Furthermore, complete $k$-uniform hypergraphs correspond to $(k,n)$-threshold access structures. Many authors~\cite{BF1, BFM, BL} have constructed secret sharing schemes for $k$-homogeneous access structures using various techniques.

\subsection{Vector Space Access Structures}
Let $\Gamma$ be an access structure defined on a set of participant $\mathcal{P}$, and let $D \not \in \mathcal{P}$ be a dealer. We say that an access structure $\Gamma$ is a {\it vector space access structure} if
there exists a function 

$$f: \mathcal{P} \cup {D} \rightarrow E\backslash \{ 0 \}$$

\noindent where $E$ is a vector space over a finite field $K$,
such that if $A \subseteq \mathcal{P}$ then $ A \in \mathcal{P}$ if and only if  the vector $f(D)$  can be expressed as a linear combination of the vectors in the set $ f(A) = \{f(p) | \ p \in A \}$.
An example of a vector space access structure is the $(k,n)$-threshold access structure,  which consists of all subsets with at least $k$ participants among the set of $n$ participants.  (See ~\cite{DS1}). The relationship between vector space access structures and ideal access structures is as follows.

\begin{theorem}\label{idealtheorem}\cite{MP2, FP2} The vector space access structures are ideal.
\begin{proof} Let  $\Gamma$ be a  vector space access structure defined on a set of participant $\mathcal{P}$, and let $D \not \in \mathcal{P}$ be a dealer. Then, there exists a function 
$f: \mathcal{P} \cup {D} \rightarrow E\backslash \{ 0 \}$, where $E$ is a vector space over a finite field $K$,
such that if $A \subseteq \mathcal{P}$ then $ A \in \mathcal{P}$ if and only if  the vector $f(D)$  can be expressed as a linear combination of the vectors in the set $ f(A) = \{f(p) | \ p \in A \}$. 
Given a secret value $k \in K$, the dealer $D$ selects  an element $v\in E$ at random such that $v\cdot f(D) = k$, where $f(D)$ is a linear combination of the vectors in the set $\{ f(p) | \ p \in D\}$. The dealer then distributes   shares $s_p = v \cdot f(p)$ to each participant $p \in \mathcal{P}$. Note that the ratio between the length of the secret and the maximum length of the shares given to the participants is one. Therefore, we conclude that $\Gamma$ is ideal.
\end{proof}
 
\end{theorem}

\subsection{Reduced Access Structures}
Let $\Gamma$ be an access structure defined on a set of participant $\mathcal{P}$.  We say that an access structure $\Gamma$ is a {\it reduced access structure} if there are no pairs of distinct participants that are equivalent.
		In relation to the access structure $\Gamma$, we define the equivalence relation $\sim$  on $\mathcal{P}$ as follows.

\begin{definition} [equivalence relation]
The two participants $a,b \in \mathcal{P}$ are said to be equivalent, denoted by $a \sim b$, if either $a=b$ or $a\neq b$ and the following two conditions are satisfied:
\begin{itemize}
\item[(i)] $\{ a, b \} \nsubseteq A $ \ \  if \ $A \in \Gamma_0$
\item[(ii)] if $ A \subset \mathcal{P}\backslash \{a, b\}$, then $A \cup \{a  \} \in \Gamma_0$ if and only if $ A \cup \{b \} \in \Gamma_0$
\end{itemize}
\noindent where $\Gamma_0$ is a family of minimal qualified subsets.
\end{definition}

Let us define the equivalence classes, induced by $\sim$, on the set of participants $\mathcal{P}$  as $\mathcal{P} / \sim = \{ [a_1], [a_2],\cdots, [a_m]\}$. Then, an access structure $\Gamma_{\sim}$ on $\mathcal{P} / \sim$ can be obtained naturally  from $\Gamma$ by identifying equivalent participants. The reduced access structure of $\Gamma$ is denoted as $\Gamma_{\sim}$, and it is isomorphic to $\Gamma ( \{a_1, a_2, \cdots, a_m \})$. Therefore, we have $\rho^* (\Gamma_{\sim}) \geq \rho^* (\Gamma)$. Furthermore, $\Gamma$ is a vector space access structure if and only if $\Gamma_{\sim}$ is as well.

\section{Lemmas}
In this section, we present several lemmas to prove the main result. The following lemma  corresponds to the $s=2$ case of Lemma 4.2, and we omit its proof. Let us define $w(Q,\Gamma)$ as the number of minimal qualified subsets contained in the set $Q$. Let $\Omega(m,\Gamma)$ be the set of possible values of $w(Q,\Gamma)$ with $|Q|=m$. This implies that  $\Omega(m,\Gamma)$ collects the numbers of minimal qualified subsets contained in any set of $m$ participants.
 
 \begin{lemma}\label{lemma1}
 Let $\Gamma$ be a $k$-homogeneous access structure defined on a set of participants $\mathcal{P}$ such that
the number of minimal qualified subsets contained in any set of $k+1$ participants is not equal to $k$, and $\rho^* (\Gamma) > \frac{k-1}{k}$.
Let us consider any $k+2$ distinct participants $u_1, u_2, \cdots, u_k, u_{k+1}, v \in \mathcal{P}$  such that
\begin{equation}\label{eqq:1111}
    \{u_1,u_2, u_{l_1}, u_{l_2}, \cdots, u_{l_{k-2}}\} \in \Gamma
\end{equation}

\noindent for all $ 3 \leq l_1, l_2, \cdots, l_{k-2} \leq k+1$.\\

 \noindent Then either
 $ w(\{u_1, u_2, u_3, \cdots, u_k, u_{k+1} \},  \ \Gamma ) = k+1 $ or
 $ \{ u_3,u_4, u_5, \cdots, u_k, u_{k+1}, v \} \not \in \Gamma $.\\
 \end{lemma}
 
 \begin{lemma}[General version of Lemma~\ref{lemma1}]\label{theorem1}
 Suppose that $\Gamma$ is a $k$-homogeneous access structure defined on a set of participants $\mathcal{P}$, such that
the number of minimal qualified subsets contained in any set of $k+1$ participants is not equal to $k$, and $\rho^* (\Gamma) > \frac{k-1}{k}$.
Let us consider any $k+2$ distinct participants $u_1, u_2, \cdots, u_k, u_{k+1}, v \in \mathcal{P}$ such that
\begin{equation}\label{eqq:1}
    \{u_1,u_2,\cdots, u_s, u_{l_1}, u_{l_2}, \cdots, u_{l_{k-s}}\} \in \Gamma
\end{equation}

\noindent for all $ s+1 \leq l_1, l_2, \cdots, l_{k-s} \leq k+1$, where $s \geq 2$.\\

 \noindent Then there exist two participants $u_i$ and $u_j$, where $ 1\leq i,j \leq k+1$, such that either
 $$ w(\{u_1, u_2, u_3, \cdots, u_k, u_{k+1} \},  \ \Gamma ) = k+1 $$ or
 $$ \{ u_{m_1},u_{m_2}, u_{m_3}, \cdots, u_{m_{k-1}}, v \} \not \in \Gamma ,$$
 where $ i, j \not \in \{m_1, m_2, m_3, \cdots, m_{k-1}\}$ and $1 \leq m_1, m_2, m_3, \cdots, m_{k-1} \leq k+1$.
\end{lemma}

 \begin{proof}
 
 \begin{sloppypar}
 Let us assume that $w(\{u_1, u_2, u_3, \cdots, u_k, u_{k+1} \},\:\Gamma) \neq k+1$ holds. This implies that the number of minimal qualified subsets contained in the set of $k+1$ participants, $\{u_1, u_2, u_3, \cdots, u_k, u_{k+1} \}$, is not equal to $k+1$.
 Also, by using the condition $k\not \in \Omega(k+1, \Gamma)$, which means that the number of minimal qualified subsets contained in any set of $k+1$ participants is not equal to $k$,   we can conclude that there exist $i$ and $j$ such that
 \begin{align} \label{eqq:2}
\{u_i, u_{m_1},u_{m_2}, u_{m_3}, \cdots, u_{m_{k-1}}\} \notin \Gamma {\text{  \  and \  \ }} 
\{u_j, u_{m_1},u_{m_2},  \cdots, u_{m_{k-1}}\} \notin \Gamma. 
  \end{align} 
  \noindent where $ i, j \not \in \{m_1, m_2, \cdots, m_{k-1}\}$, $1 \leq i,j, m_1,\cdots, m_{s-2} \leq s$, $ s+1\leq m_{s-1},\cdots,m_{k-1}  \leq k+1$.\\
  
 \noindent Note that $\{u_1, u_2, u_3, \cdots, u_k, u_{k+1}\} = \{u_i,u_j,u_{m_1},u_{m_2}, u_{m_3},\cdots, u_{m_{k-1}}\}$. Now we need to show that $\{u_{m_1}, u_{m_2}, u_{m_3}, \cdots, u_{m_{k-1}}, v\}\notin \Gamma$.  
 Let us consider the two cases as follows: 
 $\{u_i,u_{m_1}, u_{m_2}, u_{m_3}, \cdots, u_{m_{k-2}},v\}\notin \Gamma$ or $\{u_i, u_{m_1}, u_{m_2},u_{m_3}, \cdots, u_{m_{k-2}}, v\}\in \Gamma$. \\

{\bf{Case I: $\{u_i, u_{m_1}, u_{m_2},  \cdots, u_{m_{k-2}}, v \}\notin \Gamma$}}.

 In this case, let us first prove that
$\{ u_i, u_{m_1},  \cdots, u_{m_{k-2}}, u_{m_{k-1}}, v\} \notin \Gamma $. Let us assume otherwise, that is,
the set of $k+1$ different participants  $\{u_i, u_{m_1},  \cdots, u_{m_{k-2}}, u_{m_{k-1}}, v\}$ is in $\Gamma$. We can consider  the following subsets of participants $\mathcal{P}$: $B_1 = \{u_i\}$,$\:B_2=\{u_i,u_{m_1}\}$,  $\:B_3=\{u_i,u_{m_1},u_{m_2}\},\cdots,  B_{k-1}=\{u_i,u_{m_1},u_{m_2},\cdots,u_{m_{k-2}}\},$ and $\:B_k=\{u_i,u_{m_1},u_{m_2}\cdots,u_{m_{k-2}},v\}$. 
From the condition of {Case I}, we have  $ \emptyset \neq B_1 \subset B_2 \subset \cdots \subset B_k =\{u_i,u_{m_1}, u_{m_2},\cdots,u_{m_{k-2}},v\}\not \in \Gamma$. Let us consider a subset $A=\{u_j, u_{m_1}, u_{m_2}, \cdots, u_{m_{k-3}}, u_{m_{k-1}} \} \subseteq \mathcal{P}$ consisting of $k-1$ participants, where $1 \leq j, m_1,\cdots, m_{s-2} \leq s$, and $ s+1\leq m_{s-1},\cdots,m_{k-1}  \leq k+1$. Note that $A \not \in \Gamma$. We can now consider the subsets of $A$ as follows:  $X_1=\{u_j,u_{m_1},u_{m_2},\cdots,u_{m_{k-3}},u_{m_{k-1}}\},\:X_2=\{u_j,u_{m_2},u_{m_3},\cdots,u_{m_{k-3}},u_{m_{k-1}}\},\:X_3=\{u_j,u_{m_3},\cdots,u_{m_{k-3}},u_{m_{k-1}}\},\cdots,X_{k-2}=\{u_j,u_{m_{k-1}}\}$, $\:X_{k-1}=\{u_j\}$, and $\:X_k=\{u_{m_{k-1}}\}$.  \\

By using the condition (\ref{eqq:1}),  we derive that $B_1 \cup X_1 = \{u_i, u_j, u_{m_1}, \cdots,u_{m_{k-3}},u_{m_{k-1}}\}\in\Gamma$, $ \cdots,$ $ B_{k-2} \cup X_{k-2} = \{u_i, u_j, u_{m_1}, \cdots,u_{m_{k-3}}, u_{m_{k-1}}\} \in \Gamma,$ $ B_{k-1} \cup X_{k-1} = \{u_i,u_j,u_{m_1}, \cdots,u_{m_{k-3}}, u_{m_{k-2}}\} \in \Gamma$. From the assumption, 
we also observe that $B_k \cup X_k = \{u_i,u_{m_1},\cdots,u_{m_{k-2}}, u_{m_{k-1}}, v \} \in \Gamma$. Since the set of $k-1$ participants can not be in $\Gamma$, we derive that
$B_0\cup X_1=\{u_j,u_{m_1},u_{m_2},\cdots, u_{m_{k-3}},u_{m_{k-1}}\} \notin \Gamma,  B_1 \cup X_2 = \{u_i,u_j,u_{m_2},\cdots,u_{m_{k-3}},u_{m_{k-1}}\} \notin \Gamma,\cdots,  B_{k-2} \cup X_{k-1} = \{u_i, u_j,u_{m_1},\cdots,u_{m_{k-3}}\} \notin \Gamma$.
By using (\ref{eqq:2}), we can also derive that  $ B_{k-1} \cup X_k = \{u_i,u_{m_1}, u_{m_2}, \cdots, u_{m_{k-2}}, u_{m_{k-1}}\}\notin \Gamma$.\\

 Then the sequence $\phi \neq B_1 \subseteq B_2 \subseteq \cdots \subseteq B_k \notin \Gamma$ is made independent by the set $A=\{u_j, u_{m_1},\cdots,u_{m_{k-3}},u_{m_{k-1}}\} \notin \Gamma$. Therefore, using the independent sequence method, we can conclude that $\rho^*(\Gamma) \leq \frac{|A|}{k}=\frac{k-1}{k}$, which leads to  a contradiction. Hence, we conclude that $ \{u_i,u_{m_1},u_{m_2}, \cdots,u_{m_{k-2}},u_{m_{k-1}}, v \} \notin \Gamma$, which implies that $\{u_{m_1},u_{m_2},\cdots,u_{m_{k-2}}, u_{m_{k-1}}, v \} \notin \Gamma$.\\

 {\bf{Case II: $\{u_i, u_{m_1}, u_{m_2},  \cdots, u_{m_{k-2}}, v \}\in \Gamma$}}.

  In this case, we need to prove $\{u_{m_1},u_{m_2},\cdots, u_{m_{k-1}}, v \} \notin \Gamma$. Let us assume the opposite, that is, $\{u_{m_1},u_{m_2},\cdots, u_{m_{k-1}}, v \} \in \Gamma$. We can consider the following subsets of participants $\mathcal{P}$:
    $B_1 = \{u_{m_1}\},\:B_2=\{u_{m_1},u_{m_2}\},\cdots,B_{k-1}=\{u_{m_1},u_{m_2}\cdots,u_{m_{k-1}}\}$, and $B_k=\{u_j,u_{m_1},\cdots,u_{m_{k-1}}\}$. By using (\ref{eqq:2}), we have $ \emptyset \neq B_1 \subset B_2 \subset \cdots \subset B_k =\{u_j,u_{m_1},u_{m_2}\cdots,u_{m_{k-2}},u_{m_{k-1}}\}\not \in \Gamma$.
    Let us consider a subset $A=\{u_i, u_{m_2}, u_{m_3}, \cdots, u_{m_{k-3}}, u_{m_{k-2}}, v \} \subseteq \mathcal{P}$ consisting of $k-1$ participants.
    Note that $A \not \in \Gamma$. Now we can consider the subsets of $A$ as follows:
    $X_1=\{u_i,u_{m_2},\cdots,u_{m_{k-2}}, v \}$,$\:X_2=\{u_i,u_{m_3},\cdots,u_{m_{k-2}},v\}$,$\:X_3=\{u_i,u_{m_4},\cdots,u_{m_{k-2}},v\}$,
    $\cdots,X_{k-3}=\{u_i,u_{m_{k-2}},v\}$,$\:X_{k-2}=\{u_i,v\}$,$\:X_{k-1}=\{v\}$,$\:X_k=\{u_i\}$. \\

 Since $\{u_i, u_{m_1}, u_{m_2},  \cdots, u_{m_{k-2}}, v\}\in \Gamma$,  we  can observe that 
  $B_1 \cup X_1 =$ $ \{u_i,u_{m_1},u_{m_2},\cdots, u_{m_{k-2}}, v\}$ $\in\Gamma$, $\ B_2 \cup X_2=\{u_i,u_{m_1},u_{m_2},\cdots,u_{m_{k-2}},v\}$ $\in \Gamma$, $\cdots,
  \ B_{k-2} \cup X_{k-2}$ $ = \{u_i,u_{m_1},u_{m_2},\cdots,u_{m_{k-2}}, v\}\in \Gamma$. 
 Moreover, since the set of $k-1$ participants can not be in $\Gamma$, we can also observe that 
 $ B_0 \cup X_1=\{u_i,u_{m_2},\cdots,u_{m_{k-2}}, v\} \notin \Gamma, \  B_1 \cup X_2 = \{u_i,u_{m_1},u_{m_3},\cdots,u_{m_{k-2}}, v\} \notin \Gamma,\cdots, B_{k-3} \cup X_{k-2} = \{u_i,u_{m_1},\cdots,u_{m_{k-3}}, v\} \notin \Gamma,\: B_{k-2} \cup X_{k-1} = \{u_{m_1},\cdots,u_{m_{k-2}}, v\} \notin \Gamma$. From equation (\ref{eqq:2}), we can see that $B_{k-1}\cup X_k = \{u_i,u_{m_1}, u_{m_2}, \cdots, u_{m_{k-2}}, u_{m_{k-1}}\}\notin \Gamma$ and 
  $B_k = \{u_j, u_{m_1},u_{m_2}, \cdots, u_{m_{k-2}}, u_{m_{k-1}}\} \notin \Gamma$. Additionally, from the assumption, we have
  $B_{k-1} \cup X_{k-1} = \{u_{m_1},u_{m_2},\cdots, u_{m_{k-1}}, v\} \in \Gamma$. 
  By using the condition (\ref{eqq:1}),  we derive that $ \{u_i, u_j, u_{m_1}, \cdots,u_{m_{k-3}},u_{m_{k-1}}\}\in\Gamma$, which implies that $B_k \cup X_k = \{ u_i, u_j, u_{m_1}, \cdots, u_{m_{k-2}}, u_{m_{k-1}}\} \in \Gamma$.\\

  Then the sequence $\phi \neq B_1 \subseteq B_2 \subseteq \cdots \subseteq B_k \notin \Gamma$ is made independent by the set $A=\{u_i,u_{m_2}, u_{m_3}, \cdots,u_{m_{k-3}},u_{m_{k-2}}, v\} \notin \Gamma$. Hence, using the independent sequence method, we can conclude that $\rho^*(\Gamma) \leq \frac{|A|}{k}=\frac{k-1}{k}$, which leads to a contradiction. Therefore, we conclude that $ \{u_{m_1},u_{m_2}, \cdots,u_{m_{k-2}},u_{m_{k-1}}, v\} \notin \Gamma$.  
  \end{sloppypar}
  \end{proof}

 \begin{lemma}\label{lemma2}
 Let $\Gamma$ be a $k$-homogeneous access structure on a set of participants $\mathcal{P}$ such that  $\rho^* (\Gamma) > \frac{k-1}{k}$.
Let $a_1, a_2, \cdots, a_k, a_{k+1}, b_1, b_2 \in \mathcal{P}$ be $k+3$ different participants satisfying the following two conditions:
\begin{align}\label{eq:3}
 w(\{a_1, a_2, \cdots, a_k, a_{k+1} \}, \Gamma ) = k+1
 \end{align}
  \noindent and
   \begin{align} \label{eq:4}
   \{a_1, a_{i_2}, a_{i_3}, \cdots, a_{i_{k-2}}, b_1, b_2 \} \in \Gamma,
   \end{align}
   
   \noindent where  $2\leq i_2, i_3, \cdots, i_{k-2} \leq k+1.$\\

\noindent Then there exist $j_1, j_2, \cdots, j_{k-1}$, where $1 \leq j_1, j_2, \cdots, j_{k-1} \leq k+1 $, such that either 

\noindent $ \{a_{j_1}, a_{j_2}, \cdots, a_{j_{k-1}}, b_1 \} \in \Gamma$ or $\{a_{j_1}, a_{j_2}, \cdots, a_{j_{k-1}}, b_2 \} \in \Gamma$. 
\end{lemma}

 \begin{proof}
 \begin{sloppypar} Let us assume otherwise, that is, 
 \begin {align} \label{eq:5}
\{a_{j_1},\cdots, a_{j_{k-1}},b_1\} \notin \Gamma
\ \ {\text{and}}  \ \ 
\{a_{j_1},\cdots, a_{j_{k-1}},b_2\} \notin \Gamma \ \ {\text{for\ all}} \ \  1 \leq j_1,\cdots,j_{k-1} \leq k+1.
\end{align}
 We can consider 
 the following subsets of participants $\mathcal{P}$: 
  $B_1 = \{a_1\}$,\ $B_2=\{a_1,a_4\}$,\ $ B_3=\{a_1,a_4,a_5\}$, \ $\cdots$, \ $B_{k-2}=\{a_1,a_4,\cdots,a_k\}$,\ $ B_{k-1}=\{a_1,a_4,\cdots,a_k,b_1\},$ and $B_k=\{a_1,a_4,\cdots,a_k,a_{k+1},b_1\}$. Note that  $ \emptyset \neq B_1 \subset B_2 \subset \cdots \subset B_k =\{a_1,a_4,a_5\cdots,a_k, a_{k+1},b_1\}$. Let us consider a subset $A=\{a_2, a_3, a_4, \cdots, a_{k-2}, b_1, b_2 \} \subseteq \mathcal{P}$ consisting of $k-1$ participants. Note that $A \not \in \Gamma$. We can now consider the subsets of $A$ as follows:  $X_1=\{a_2,a_3,a_4,\cdots,a_{k-2},b_1,b_2\},\:X_2=\{a_2,a_3,a_5,\cdots,a_{k-2},b_1,b_2\},\:X_3=\{a_2,a_3,a_6,\cdots,a_{k-2},b_1,b_2\},\cdots,X_{k-3}=\{a_2,b_1,b_2\},\:X_{k-2}=\{b_1,b_2\},\:X_{k-1}=\{b_2\},\:X_k=\{a_3\}$.  \\

\noindent By using the condition (\ref{eq:4}), we can derive that 
  $B_1 \cup X_1 = \{a_1,a_2,a_3,\cdots,a_{k-2},b_1,b_2\}\in\Gamma$, $B_2 \cup X_2=\{a_1,a_2,a_3, \cdots,a_{k-2},b_1,b_2\}\in \Gamma,$ $\cdots,$ $B_{k-2} \cup X_{k-2}=\{a_1,a_4,a_5, \cdots,a_{k},b_1,b_2\}\in \Gamma.$ Moreover, since the set of $k-1$ participants can not be in  $\Gamma$, we also observe that
  $ B_0 \cup X_1=\{a_2,a_3,\cdots,a_{k-2},b_1,b_2\} \notin \Gamma$, \ $ B_1 \cup X_2 = \{a_1,a_2,a_3,a_5,\cdots,a_{k-2},b_1,b_2\} \notin \Gamma$,$\cdots$, \ $ B_{k-2} \cup X_{k-1} = \{a_1,a_4,\cdots,a_k,b_2\} \notin \Gamma$. Furthermore, using the assumption (\ref{eq:5}), we obtain 
  $B_{k-1} \cup X_k = \{a_1,a_3,\cdots,a_k,b_1\} \notin \Gamma$.
  Additionally, using the condition (\ref{eq:4}), we have
   $B_{k-1} \cup X_{k-1} = \{a_1,a_4, a_5, \cdots,a_k,b_1,b_2\}\in \Gamma$. By using  (\ref{eq:3}), we derive that 
  $\{a_1, a_3, a_4, \cdots, a_k, a_{k+1}\} \in \Gamma$, which implies that $B_k \cup X_k = \{a_1,a_3,a_4,\cdots, a_k, a_{k+1},b_1\} \in \Gamma$. \\

\noindent Then the sequence $\phi \neq B_1 \subseteq B_2 \subseteq \cdots \subseteq B_k \notin \Gamma$ is made independent by the set $A=\{a_2,a_3,a_4,\cdots,a_{k-2},b_1,b_2\}$, which is also not in $\Gamma$. Therefore, using the independent sequence method, we can conclude that $\rho^*(\Gamma) \leq \frac{|A|}{k}=\frac{k-1}{k}$.  This leads to  a contradiction. Hence, we conclude that   
there exist $j_1, j_2, \cdots, j_{k-1}$, where $1 \leq j_1, j_2, \cdots, j_{k-1} \leq k+1 $, such that either  
  \begin {align*} 
\{a_{j_1},\cdots, a_{j_{k-1}},b_1\} \in \Gamma
\ \ {\text{or}}  \ \ 
\{a_{j_1},\cdots, a_{j_{k-1}},b_2\} \in \Gamma.
\end{align*}

\noindent  This completes the proof of Lemma~\ref{lemma2}.  

 \end{sloppypar}
 \end{proof}

  \begin{lemma}\label{lemma3}
 Assume that $\Gamma$ is  a $k$-homogeneous access structure on a set of participants $\mathcal{P}$ such that
the number of minimal qualified subsets contained in any set of $k+1$ participants is not equal to $k$, and $\rho^* (\Gamma) > \frac{k-1}{k}$.
Let us consider $k+2$ different participants $a_1, a_2, \cdots, a_k, a_{k+1}, b \in \mathcal{P}$ satisfying the following three conditions:

\begin{align}\label{eq:8}
w(\{a_1, a_2, \cdots, a_k, a_{k+1} \}, \Gamma ) = k+1,
\end{align}
\noindent and
\begin{align}\label{eq:9}
 \{a_1, a_2, a_{i_1}, a_{i_2},  \cdots, a_{i_{k-3}}, b \} \in \Gamma\:\:\text{where}\:\:3\leq i_1,i_2,  \cdots, i_{k-3} \leq k+1,
\end{align}
\noindent and
\begin{align}\label{eq:10}
 \{a_{t_1}, a_{t_2},  \cdots, a_{t_{k-3}}, a_k, a_{k+1}, b \} \in \Gamma \:\:\text{where}\:\: 1 \leq t_1, t_2,  \cdots, t_{k-3} \leq k-1.
\end{align}

\noindent Then, the induced access structure $\Gamma(\{a_1, a_2, \cdots, a_k, a_{k+1}, b \})$ is the
$(k,k+2)$-threshold access structure.

\end{lemma}

 \begin{proof}
 
 \begin{sloppypar}
 
 To prove this, we need to show that $w(\{ a_{p_1}, a_{p_2}, a_{p_3}, \cdots, a_{p_k}, b \}, \Gamma) = k+1$, where $ 1 \leq p_1, p_2, \cdots, p_k \leq k+1 $. We will prove it  by  showing the following two claims.

 \begin{claim*}\label{claim41}
$w(\{a_1,a_2,a_{q_1},a_{q_2},\cdots,a_{q_{k-2}},b \}, \Gamma) = k+1$, where $3\leq q_1,q_2,  \cdots, q_{k-2} \leq k+1$.
\end{claim*}
 
 \begin{proof}[Proof of Claim~\ref{claim41}]
 \noindent Using the condition (\ref{eq:8}), we obtain 
 \begin{align}\label{eq:11}
  \{ a_1,a_2,a_{q_1},a_{q_2},\cdots,a_{q_{k-2}}\} \in \Gamma,
  \end{align} where  $3\leq q_1,q_2,  \cdots, q_{k-2} \leq k+1$.\\
  
 \noindent Moreover, using the condition (\ref{eq:9}),  we also obtain
 \begin{align}\label{eq:12}
 \{a_1, a_2, a_{i_1}, a_{i_2},  \cdots, a_{i_{k-3}}, b \} \in \Gamma
 \end{align}
 
 \noindent where $ i_1,i_2,  \cdots, i_{k-3} \in \{q_1,q_2,\cdots, q_{k-2}\}$ with $3\leq q_1,q_2,  \cdots, q_{k-2} \leq k+1$. \\

 Thus, we need to show that $\{a_2,a_{q_1},a_{q_2}, \cdots,a_{q_{k-2}}, b\} \in \Gamma$ or 
 $\{ a_1,a_{q_1},a_{q_2}, \cdots,a_{q_{k-2}}, b\} \in \Gamma$, where  $3\leq q_1,q_2,  \cdots, q_{k-2} \leq k+1$.
 Suppose, for the sake of contradiction, that  $\{a_2,a_{q_1},a_{q_2}, \cdots,a_{q_{k-2}}, b\} \not \in \Gamma$ and  $\{ a_1,a_{q_1},a_{q_2},\cdots,a_{q_{k-2}}, b \} \not \in \Gamma$, where  $3\leq q_1,q_2,  \cdots, q_{k-2} \leq k+1$. Now we consider  the subsets of participants $\mathcal{P}$ as follows: $B_1 = \{a_{q_{k-2}}\},\:B_2=\{a_{q_{k-2}}, b \}$,  $\:B_3=\{a_{q_{k-3}},a_{q_{k-2}}, b \},\cdots,\:B_{k-1}=\{a_{q_1},a_{q_2},\cdots,a_{q_{k-2}}, b \},$ and $\:B_k=\{a_2,a_{q_1},a_{q_2}, a_{q_3}, \cdots,a_{q_{k-2}},b\}$. 
From the given assumption, we have  a sequence of sets $ \emptyset \neq B_1 \subset B_2 \subset \cdots \subset 
\:B_k=\{a_2,a_{q_1},a_{q_2},  \cdots,a_{q_{k-2}},b\}
\not \in \Gamma$. Let us choose a participant, denoted by $a_{q_{k-1}}$,  from the set of  $k+1$ participants $a_1,a_2, \cdots, a_{k+1}$ who is not in the set  $\{a_1, a_2, a_{q_1}, a_{q_2}, \cdots, a_{q_{k-2}}\}$. Next, we consider
a subset $A=\{a_1,a_2, a_{q_2},  \cdots, a_{q_{k-3}}, a_{q_{k-1}} \} \subseteq \mathcal{P}$ consisting of $k-1$ participants. Note that $A \not \in \Gamma$. Now, we consider the subsets of $A$ as follows:  $X_1=\{a_1,a_2, a_{q_2}, a_{q_3}, \cdots, a_{q_{k-3}}, a_{q_{k-1}}\},\:X_2=\{a_1,a_2, a_{q_2}, a_{q_3}, \cdots, a_{q_{k-4}}, a_{q_{k-1}}\},\cdots,X_{k-3}= \{a_1,a_2, a_{q_{k-1}}\}, X_{k-2}=\{a_1,a_2\}$, $\:X_{k-1}=\{a_{q_{k-1}}\}$, and $\:X_k=\{a_1\}$.  \\

Using condition (\ref{eq:8}),   we observe 
 $B_1 \cup X_1 = \{a_1,a_2, a_{q_{2}}\cdots,a_{q_{k-3}},a_{q_{k-2}},a_{q_{k-1}}\}\in\Gamma$. Additionally, using condition (\ref{eq:9}), we obtain that $\ B_2 \cup X_2$ $ = \{a_1,a_2, a_{q_{2}}\cdots,a_{q_{k-4}},a_{q_{k-2}},a_{q_{k-1}}, b \}\in\Gamma$, $\cdots$,
$\ B_{k-2} \cup X_{k-2} = \{a_1,a_2, a_{q_{2}}\cdots,a_{q_{k-3}},a_{q_{k-2}}, b \}\in\Gamma$. Also, we have 
$\ B_{k-1} \cup X_{k-1} =$ $ \{a_{q_1}, a_{q_{2}}\cdots,a_{q_{k-2}},a_{q_{k-1}}, b \} =  \{a_3,a_4,a_5, \cdots,a_k,a_{k+1},b\} \in\Gamma$. 
By using the condition (\ref{eq:9}), we also deduce that $\{a_1,a_2, a_{r_1}, a_{r_{2}}\cdots,a_{r_{k-3}}, b \}  \in \Gamma$, where $a_{r_1},a_{r_2},\cdots, a_{r_{k-3}} \in \{a_{q_1}, a_{q_{2}}\cdots,a_{q_{k-2}}\}$, which implies that
$\ B_{k} \cup X_{k} = \{a_1,a_2, a_{q_1}, a_{q_{2}}\cdots,a_{q_{k-2}}, b \}  \in\Gamma$. \\

Since the set of $k-1$ participants can not be in $\Gamma$, we derive that
$B_0\cup X_1=\{a_1,a_2,a_{q_2},a_{q_3},\cdots, a_{q_{k-3}},a_{q_{k-1}}\} \notin \Gamma,$ $ \  B_1 \cup X_2 = \{a_1,a_2,a_{q_2},\cdots,a_{q_{k-4}},a_{q_{k-2}},a_{q_{k-1}}\} \notin \Gamma,$ $\cdots,  B_{k-2} \cup X_{k-1} = \{a_{q_2},\cdots,a_{q_{k-2}},a_{q_{k-1}}, b\} \notin \Gamma$.
 From the given assumption, we get $B_{k-1} \cup X_{k} = \{a_1,a_{q_2},\cdots,a_{q_{k-2}}, b\} \notin \Gamma$.\\

 Then the sequence $\phi \neq B_1 \subseteq B_2 \subseteq \cdots \subseteq B_k \notin \Gamma$ is made independent by the set $A=\{a_1,a_2, a_{q_2},  \cdots, a_{q_{k-3}}, a_{q_{k-1}} \}$. Therefore, using the independent sequence method, we can obtain that $\rho^*(\Gamma) \leq \frac{|A|}{k}=\frac{k-1}{k}$, which leads to  a contradiction. Hence, we conclude that 
 $\{a_2,a_{q_1},a_{q_2}, \cdots,a_{q_{k-2}}, b\} \in \Gamma$ or $\{ a_1,a_{q_1},a_{q_2}, \cdots,a_{q_{k-2}}, b\} \in \Gamma$, where  $3\leq q_1,q_2,  \cdots, q_{k-2} \leq k+1$. 
By using the condition $k\not \in \Omega(k+1, \Gamma)$, we derive that
 $\{a_2,a_{q_1},a_{q_2},\cdots,a_{q_{k-2}}, b\} \in \Gamma$ and $\{ a_1,a_{q_1},a_{q_2}, \cdots,a_{q_{k-2}}, b\} \in \Gamma$, where  $3\leq q_1,q_2,  \cdots, q_{k-2} \leq k+1$. Therefore, we conclude that
 $w(\{a_1,a_2,a_{q_1},a_{q_2},\cdots,a_{q_{k-2}},b \}, \Gamma) = k+1$, where $3\leq q_1,q_2,  \cdots, q_{k-2} \leq k+1$.\\ 
\end{proof}

 \begin{claim*}\label{claim42}
$w(\{a_{u_1},a_{3},a_{4}, a_{5},\cdots,a_{k},a_{k+1},b \}, \Gamma) = k+1$, where $u_1 = 1$ or $2$.
\end{claim*}
 \begin{proof}[Proof of Claim~\ref{claim42}] 
 \noindent By using the condition (\ref{eq:10}),  we derive 
 \begin{align}\label{eq:12}
 \{a_{j_1}, a_{j_2}, \cdots,a_{j_{k-3}}, a_{k}, a_{k+1}, b \} \in \Gamma
 \end{align}
 
 \noindent where $ j_1,i_2,  \cdots, j_{k-3} \in \{u_1,3,4,5,\cdots,k-1\}$ with $u_1=1$ or $2$.\\

Thus, we need to prove that $\{a_{u_1},a_{3},\cdots,a_{k-1},a_{k+1}, b\} $ $\in \Gamma$ or $\{a_{u_1},a_{3},\cdots,a_{k-1},a_{k}, b\} \in \Gamma$, where  $u_1 = 1$ or $2$.
Suppose, for the sake of contradiction, that $\{a_{u_1},a_{3},a_{4},\cdots,a_{k-1},a_{k+1}, b\} \not \in \Gamma$ and $\{a_{u_1},a_{3},a_{4},\cdots,a_{k-1},a_{k}, b\} \not \in \Gamma$, where  $u_1=1 {\text{\ or }} 2$.
Let us consider  the following subsets of participants $\mathcal{P}$: $B_1 = \{a_{u_1}\},\:B_2=\{a_{u_1}, b \}$,  $\:B_3=\{a_{u_1},a_{3}, b \},\cdots, B_{k-2}=\{a_{u_1},a_{3},\cdots,a_{k-2}, b \},\:B_{k-1}=\{a_{u_1},a_{3},\cdots,a_{k-1}, b \},$ and $\:B_k=\{a_{u_1},a_{3}, \cdots,a_{k-1},a_{k+1},b\}$, where  $u_1=1 {\text{\ or }} 2$.
From the given assumption, we have  $ \emptyset \neq B_1 \subset B_2 \subset \cdots \subset 
\:B_k= \{a_{u_1},a_{3}, \cdots,a_{k-1},a_{k+1},b\}
\not \in \Gamma$,  where  $u_1=1 {\text{\ or }} 2$.\\

 Let us choose a participant, denoted by $u_2$, from the set of  two participants $a_1,a_2$ who is not equal to $u_1$. Next, we consider
a subset $A=\{a_{u_2},a_4, a_{5},  \cdots, a_{k}, a_{k+1} \} \subseteq \mathcal{P}$ consisting of $k-1$ participants. Note that $A \not \in \Gamma$. We can now consider the subsets of $A$ as follows:  $X_1=\{a_{u_2},a_4, a_{5},\cdots, a_{k}, a_{k+1}\},\:X_2=\{a_{u_2},a_5, \cdots, a_{k}, a_{k+1}\},\cdots,X_{k-3}= \{a_{u_2},a_k,a_{k+1}\}, X_{k-2}=\{a_k,a_{k+1}\}$, $\:X_{k-1}=\{a_{u_2}\}$, and $\:X_k=\{a_k\}$.  
By using the condition (\ref{eq:8}),   we observe that $B_1 \cup X_1 = \{a_{u_1},a_{u_2},a_4, a_5, \cdots,a_k, a_{k+1}\}= \{a_{1},a_{2},a_4, a_5, \cdots,a_k, a_{k+1}\}\in\Gamma$. Moreover, using the condition  (\ref{eq:9}) and (\ref{eq:10}), we also derive that $\ B_2 \cup X_2 =
\{a_{u_1},a_{u_2},a_5,  \cdots, a_{k}, a_{k+1}, b\} \in\Gamma$, $\cdots$,
  $\ B_{k-2} \cup X_{k-2} =\{a_{u_1},a_{3},a_{4},\cdots,a_{k-2},a_k,a_{k+1}, b\} \in \Gamma$, and
$\ B_{k-1} \cup X_{k-1} = \{a_{u_1}, a_{u_{2}}, a_3,\cdots,a_{k-1}, b \} =  
\{a_{1}, a_{2}, a_3,\cdots,a_{k-1}, b \}
 \in\Gamma$. 
 Additionally, using (\ref{eq:10}), we can derive that $\{a_{u_1},a_{r_1}, a_{r_{2}}\cdots,a_{r_{k-4}},a_k,a_{k+1}, b \}  \in \Gamma$, where $a_{r_1},a_{r_2},\cdots, a_{r_{k-4}} \in \{a_{3}, a_{4}\cdots,a_{k-1}\}$, which implies that $\ B_{k} \cup X_{k} = \{a_{u_1},a_3, a_{4}, \cdots,a_{k-1}, a_k, a_{k+1}, b \}  \in\Gamma$. Since the set of $k-1$ participants can not be in $\Gamma$, we deduce  that
$B_0\cup X_1=\{a_{u_2},a_4,a_5,\cdots, a_k,a_{k+1}\} \notin \Gamma, \  B_1 \cup X_2 = \{a_{u_1},a_{u_2},a_5,\cdots,a_{k},a_{k+1}\} = \{a_{1},a_{2},a_5,\cdots,a_{k},a_{k+1}\} \notin \Gamma,\cdots,  B_{k-2} \cup X_{k-1} = \{a_{u_1},a_{u_2},a_3,\cdots,a_{k-2}, b\} 
= \{a_{1},a_{2},a_3,\cdots,a_{k-2}, b\}\notin \Gamma$.
From the given assumption, we have $B_{k-1} \cup X_{k} = \{a_{u_1},a_{3},\cdots,a_{k-1}, a_k, b\} \notin \Gamma$.\\

 Then the sequence $\phi \neq B_1 \subseteq B_2 \subseteq \cdots \subseteq B_k \notin \Gamma$ is made independent by the set $A= \{a_{u_2},a_4, a_{5},  \cdots, a_{k}, a_{k+1} \}$. Hence, using the independent sequence method,  we can conclude that $\rho^*(\Gamma) \leq \frac{|A|}{k}=\frac{k-1}{k}$, which leads to   a contradiction. Therefore, we conclude that 
 $\{a_{u_1},a_{3},a_{4},\cdots,a_{k-1},a_{k+1}, b\} \in \Gamma$ or $\{a_{u_1},a_{3},a_{4},\cdots,a_{k-1},a_{k}, b\} \in \Gamma$, where  $u_1 = 1$ or $2$. 
By using the condition $k\not \in \Omega(k+1, \Gamma)$, we derive that
 $\{a_{u_1},a_{3},a_{4},\cdots,a_{k-1},a_{k+1}, b\} \in \Gamma$ and $\{a_{u_1},a_{3},a_{4},\cdots,a_{k-1},a_{k}, b\} \in \Gamma$, where  $u_1 = 1$ or $2$.
 Therefore, we conclude that
 $w(\{a_{u_1},a_3,a_{4},a_{5},\cdots,a_k, a_{k+1},b \}, \Gamma) = k+1$, where  $u_1 = 1$ or $2$. 
 \end{proof} 

\noindent Using Claim~\ref{claim41} and Claim~\ref{claim42}, we  can conclude that $$w(\{ a_{p_1}, a_{p_2}, a_{p_3}, \cdots, a_{p_k}, b \}, \Gamma) = k+1,$$
 where $ 1 \leq p_1, p_2, \cdots, p_k \leq k+1 $. \\
 
 \noindent Thus,  the induced access structure $\Gamma(\{a_1, a_2, \cdots, a_k, a_{k+1}, b \})$ is the
$(k,k+2)$-threshold access structure. This completes the proof of Lemma~\ref{lemma3}.
 \end{sloppypar}
 \end{proof}
 
 Before stating the next lemma, we need to introduce the following notation.
Let  $p$ and $q$ be two participants from the set of all participants  $\mathcal{P}$.  We say that $p$ and $q$ are  {\it{equivalent}} if either (i) $p=q$ or (ii) $p\neq q$ and the following two conditions are satisfied:
(1) $\{p, q \} \not \subset A$ if $A\in \Gamma_0$,  and (2) if $A \subset \mathcal{P} \backslash \{p,q\}$, then
$ A \cup \{p\} \in \Gamma_0$ if and only if $A \cup \{q\} \in \Gamma_0$,  where $\Gamma_0$ is the collection of minimal qualified subsets.

 \begin{lemma}\label{lemma4}
 Let $\Gamma$ be a $k$-homogeneous access structure on a set of participants $\mathcal{P}$, such that
the number of minimal qualified subsets contained in any set of $k+1$ participants is not equal to  $1$ and $k$, and $\rho^* (\Gamma) > \frac{k-1}{k}$.
Let us consider $k+2$ different participants  $a_1, a_2, \cdots, a_k, a_{k+1}, b \in \mathcal{P}$ satisfying the following conditions:
 
\begin{align}
&(i) \ \ \  w(\{a_1, a_2, \cdots, a_k, a_{k+1} \}, \Gamma ) = k+1 \label{eq:15}\\ \nonumber & \\
&(ii)\ \ \  \{a_1, a_2, \cdots, a_{k-1}, b \} \in \Gamma.\label{eq:16}\\ \nonumber & \\
&(iii) \ \ \ a_k {\text \ is \ not \ equivalent \ to \ }  a_{k+1} \label{eq:166}\\ \nonumber
\end{align}

\noindent Moreover,  $k+2$ participants $a_1, a_2, \cdots, a_k, a_{k+1}, b$ satisfy the one condition among the following  three statements.  

\begin{align}
(iv) & \ \ \  \{a_{t_1}, a_{t_2}, a_{t_3}, \cdots, a_{t_{k-3}}, a_k, a_{k+1}, b \} \not \in \Gamma \ \  \ where   \  \ 1 \leq t_1, t_2,  \cdots, t_{k-3} \leq k-1 .\label{eq:17}\\ & \nonumber \\
\noindent & or \  \{a_{l_1}, a_{l_2}, a_{l_3}, \cdots, a_{l_{k-2}}, a_{k+1}, b \} \not \in \Gamma \ \ \ where \  \  1 \leq l_1, l_2,  \cdots, l_{k-2} \leq k-1. \label{eq:18}\\ & \nonumber \\
\noindent & or  \  \{a_{s_1}, a_{s_2}, a_{s_3}, \cdots, a_{s_{k-2}}, a_{k}, b \} \not \in \Gamma \ \ \  where  \ \ 1 \leq s_1, s_2,  \cdots, s_{k-2} \leq k-1. \label{eq:19}\\ \nonumber
\end{align}

\noindent Then, either $b$ is equivalent to $a_{k+1}$ or $b$ is equivalent to $a_{k}$. 

\end{lemma}

 \begin{proof} 
First, let us consider $k+2$ participants in $\mathcal{P}$, denoted as $a_1, a_2, \cdots, a_k, a_{k+1}$, and $ b$, who satisfy conditions (\ref{eq:15}), (\ref{eq:16}), (\ref{eq:166}), and (\ref{eq:17}). Now, let us focus on  $k+1$ participants
 $a_{\Delta_1}, a_{\Delta_2}, \cdots a_{\Delta_{k-2}}, a_k, a_{k+1}, b $, where $ 1\leq \Delta_1, \cdots, \Delta_{k-2} \leq k-1$. Using (\ref{eq:15}), we obtain
 \begin{align}\label{eq:20}
 \{a_{\Delta_1}, a_{\Delta_2}, \cdots a_{\Delta_{k-2}}, a_k, a_{k+1}\} \in \Gamma.
 \end{align}


 \noindent Moreover, by (\ref{eq:17}), we derive that
 \begin{align}\label{eq:21}
 \{ a_{\beta_1}, a_{\beta_2}, \cdots a_{\beta_{k-3}}, a_k, a_{k+1}, b \} \not \in \Gamma,
 \end{align}

 \noindent where, $ a_{\beta_1}, a_{\beta_2}, \cdots a_{\beta_{k-3}} \in \{ a_{\Delta_1}, a_{\Delta_2}, \cdots a_{\Delta_{k-2}}\}.$\\
  
 \noindent By using the conditions  $1 \not \in \Omega(k+1, \Gamma)$ and $a_k$ is not equivalent to $a_{k+1}$, we obtain that
 
 \begin{align}\label{eq:23}
& \{ a_{\Delta_1}, a_{\Delta_2},  \cdots a_{\Delta_{k-2}}, a_{k+1}, b \} \not \in \Gamma 
\  \   \   \    \ or    & \{ a_{\Delta_1}, a_{\Delta_2},  \cdots a_{\Delta_{k-2}}, a_{k}, b \} \not \in \Gamma,\end{align}
where $ 1\leq \Delta_1, \Delta_2 \cdots, \Delta_{k-2} \leq k-1$.\\
 
\noindent {\bf{ Case I: $ \{ a_{\Delta_1}, a_{\Delta_2}, \cdots a_{\Delta_{k-2}}, a_{k+1}, b \} \not \in \Gamma$, where $  1 \leq \Delta_1, \Delta_2, \cdots, \Delta_{k-2} \leq k-1.$}} \\

\noindent In this case, we  aim to demonstrate the equivalence of $a_{k+1}$  and $b$. Given that  $a_{k+1}$ and $b$ represent distinct participants, it is necessary to establish that:

\begin{itemize}
\item[(1)] $\{a_{k+1}, b\} \not \subset A$ for any $A \in \Gamma_0$,
\item[(2)] if $A \subset \mathcal{P}\backslash \{ a_{k+1}, b \}$, then $A \cup \{a_{k+1}\}\in \Gamma$ if and only if $A \cup \{ b \} \in \Gamma$,
\end{itemize}

\noindent where $\Gamma_0$ is the collection of minimal qualified subsets.\\

\noindent Based on our assumption of Case I, we observe that 
\begin{equation}\label{eqn1}
\{a_{l_1}, a_{l_2}, \cdots, a_{l_{k-2}}, a_{k+1}, b \} \not \in \Gamma,
\end{equation}
where $ 1 \leq l_1, l_2, \cdots, l_{k-2}\leq k-1$. \\

\noindent Since $a_k$ is not equivalent to $a_{k+1}$, we can also deduce that
\begin{equation}\label{eqn2}
\{a_{l_1}, a_{l_2}, \cdots, a_{l_{k-2}}, a_{k}, b \}\in\Gamma,
\end{equation}
 where $ 1 \leq l_1, l_2, \cdots, l_{k-2}\leq k-1$.\\
 
 \noindent  Using the property~(\ref{eq:20}), we also obtain that 
 \begin{equation}\label{eqn3}
 \{a_{l_1}, a_{l_2}, \cdots, a_{l_{k-2}}, a_{k},a_{k+1} \}  \in \Gamma, 
 \end{equation}
 where $ 1 \leq l_1, l_2, \cdots, l_{k-2}\leq k-1$.  \\
 
\noindent To prove the property $(1)$, we will first establish the following claims.

\begin{claim}\label{claim53}
$\{a_{l_1},a_{l_2}, \cdots, a_{l_{k-3}}, x, a_{k+1}, b \} \not \in \Gamma$ and $\{a_{l_1},a_{l_2}, \cdots, a_{l_{k-3}}, x, a_{k}, b \}  \in \Gamma$, where $1 \leq l_1, l_2,\cdots, l_{k-3} \leq k-1$ and $x \in \mathcal{P}\backslash \{a_1, a_2,  \cdots, a_{k}, a_{k+1}, b \}$.
\end{claim}
  \begin{proof}[Proof of Claim~\ref{claim53}]
  Assuming the contrary, let us consider $\{a_{l_1},a_{l_2}, \cdots, a_{l_{k-3}}, x, a_{k+1}, b \}  \in \Gamma$, where $1 \leq l_1, l_2, \cdots, l_{k-3} \leq k-1$ and $x \in \mathcal{P}\backslash \{a_1, a_2, \cdots, a_{k}, a_{k+1}, b \}$.\\

Next, let us select a participant from $\{a_{1},a_{2}, \cdots, a_{k-1}\}$ who is not part of $\{a_{l_1},a_{l_2}, \cdots, a_{l_{k-3}}\}$ and denote this participant as $a_{l_{k-2}}$. We will now consider the following subsets of participants in $\mathcal{P}$:  $B_1 = \{a_{l_1}\},\:B_2=\{a_{l_1}, b \}$,  $\:B_3=\{a_{l_1},a_{l_2}, b \},\cdots, B_{k-2}=\{a_{l_1},a_{l_2},\cdots,a_{l_{k-3}}, b \},\:B_{k-1}=\{a_{l_1},a_{l_2}\cdots,a_{l_{k-2}}, b \},$ and $\:B_k=\{a_{l_1},a_{l_2}, \cdots,a_{l_{k-2}}, b, x\}$. 
Note that  $ \emptyset \neq B_1 \subset B_2 \subset \cdots \subset 
\:B_k=\{a_{l_1},a_{l_2}, \cdots,a_{l_{k-2}},  b, x\}
$.  \\

 Now, let us consider
a subset $A=\{a_{l_2},a_{l_3},  \cdots, a_{l_{k-2}}, a_k , a_{k+1}\} \subseteq \mathcal{P}$ consisting of $k-1$ participants. Note that $A \not \in \Gamma$. We can then consider the subsets of $A$ as follows:   $X_1=\{a_{l_2},a_{l_3},   \cdots, a_{l_{k-2}}, a_k, a_{k+1}\}$,\ \ $X_2=\{a_{l_2},a_{l_3},  \cdots, a_{l_{k-2}}, a_{k}\},\:X_3=
\{a_{l_{3}},a_{l_{4}},  \cdots, a_{l_{k-2}}, a_{k}\},\cdots,$ $X_{k-3}= \{a_{l_{k-3}},a_{l_{k-2}}, a_{k}\},$ $ X_{k-2}=\{a_{l_{k-2}},a_{k}\}$, $\:X_{k-1}=\{a_{k}\}$, and $\:X_k=\{a_{k+1}\}$.  \\

\noindent Using~\ref{eqn3},  we observe  that $B_1 \cup X_1 = \{a_{l_1}, a_{l_2}, \cdots, a_{l_{k-2}}, a_{k},a_{k+1} \} \in\Gamma$. Using~\ref{eqn2}, we also derive that $\ B_2 \cup X_2   = \{a_{l_1}, a_{l_2}, \cdots, a_{l_{k-2}}, a_{k}, b \} \in\Gamma,$
 and so on up to
$\ B_{k-1} \cup X_{k-1} = \{a_{l_1},a_{l_2},\cdots, a_{l_{k-2}}, a_{k}, b\} \in\Gamma$. 
 Based on the assumption, we have $ \{a_{l_1},a_{l_2},\cdots,a_{l_{k-3}},  a_{k+1}, b, x\}  \in\Gamma$, which implies that
$\ B_{k} \cup X_{k} = \{a_{l_1},a_{l_2},\cdots,a_{l_{k-2}}, a_{k+1}, b, x \}  \in\Gamma$. 
 Since the set of $k-1$ participants can not be in $\Gamma$, we derive that
$B_0\cup X_1 \notin \Gamma, \  B_1 \cup X_2  \notin \Gamma,\cdots,  B_{k-2} \cup X_{k-1} 
\notin \Gamma$.
Using~\ref{eqn1}, we also derive that $B_{k-1} \cup X_{k} = \{a_{l_1}, a_{l_2}, \cdots, a_{l_{k-2}}, a_{k+1}, b \}  \not \in \Gamma$. \\

If $\:B_k=\{a_{l_1},a_{l_2}, \cdots,a_{l_{k-3}},a_{l_{k-2}}, b, x\} \not \in \Gamma$, then the sequence $\phi \neq B_1 \subseteq B_2 \subseteq \cdots \subseteq B_k \notin \Gamma$ is made independent by the set $A= \{a_{l_2},a_{l_3}, \cdots, a_{l_{k-2}},  a_k, a_{k+1} \}$. Hence, using the independent sequence method, we can conclude that $\rho^*(\Gamma) \leq \frac{|A|}{k}=\frac{k-1}{k}$, which leads to a contradiction. Therefore, we can conclude that 
\begin{align}\label{eq:27}
\{a_{l_1},a_{l_2}, \cdots,a_{l_{k-3}},a_{l_{k-2}}, b, x\}  \in \Gamma,
\end{align}
\noindent where $1 \leq l_1,l_2,l_3,\cdots, l_{k-3},l_{k-2} \leq k-1$.\\

\noindent Let us consider $k+2$ distinct participants: $a_{l_1},a_{l_2}, \cdots,a_{l_{k-2}}, a_k, a_{k+1}, x, b \in \mathcal{P}$, where $1 \leq l_1,l_2,\cdots, l_{k-2} \leq k-1$ and $x \in \mathcal{P}\backslash \{a_1, a_2,  \cdots, a_{k}, a_{k+1}, b \}$. Next, we can apply Lemma~\ref{lemma1} with $u_1=b, u_2=x$, and $v=a_k$.  Then either

 $$ w(\{a_{l_1},a_{l_2}, \cdots,a_{l_{k-2}}, a_{k+1}, b, x \},  \ \Gamma ) = k+1, {\text {\ where\ }}  1 \leq l_1,l_2,\cdots, l_{k-2} \leq k-1,$$ 
 or
 $$ \{ a_{l_1},a_{l_2},  \cdots, a_{l_{k-2}},  a_{k+1}, a_k \} \not \in \Gamma, {\text {\ for \ }} a_k\in \mathcal{P}\backslash \{a_{l_1},a_{l_2},\cdots, a_{l_{k-2}}, a_{k+1}, b \}.$$\\

\noindent Since 
$\{a_{l_1}, a_{l_2}, \cdots, a_{l_{k-2}}, a_{k+1}, b \} \not \in \Gamma$, where $ 1 \leq l_1, l_2, \cdots, l_{k-2}\leq k-1$, we can conclude that $$w(\{a_{l_1},a_{l_2}, \cdots,a_{l_{k-2}}, a_{k+1}, b, x \},  \ \Gamma ) \neq k+1, {\text {\ where\ }}  1 \leq l_1,l_2,l_3,\cdots, l_{k-3},l_{k-2} \leq k-1.$$

\noindent  Let us consider $k+1$ participants:  $a_{l_1},a_{l_2}, \cdots,a_{l_{k-2}}, a_k, a_{k+1},  b \in \mathcal{P}$, where $1 \leq l_1,l_2,\cdots, l_{k-2} \leq k-1$. Since the number of minimal qualified subsets contained in any set of $k+1$ participants is not equal to  $1$, we must have 
$$ \{ a_{l_1},a_{l_2},  \cdots, a_{l_{k-2}},  a_{k+1}, a_k \}  \in \Gamma, {\text {\ for \ }} a_k\in \mathcal{P}\backslash \{a_{l_1},a_{l_2}, \cdots, a_{l_{k-2}}, a_{k+1}, b \},$$

\noindent which leads to a contradiction. Therefore, we can conclude that 

$$\{a_{l_1},a_{l_2}, \cdots, a_{l_{k-3}}, x, a_{k+1}, b \} \not \in \Gamma,$$ where $1 \leq l_1, l_2,\cdots, l_{k-3} \leq k-1$ and $x \in \mathcal{P}\backslash \{a_1, a_2,  \cdots, a_{k}, a_{k+1}, b \}.$ \\

 \noindent Since  $a_k$ is not equivalent to $a_{k+1}$, we also conclude that
 $$
\{a_{l_1},a_{l_2}, \cdots, a_{l_{k-3}}, x, a_{k}, b \}  \in \Gamma,
$$
\noindent where $1 \leq l_1, l_2,\cdots, l_{k-3} \leq k-1$ and $x \in \mathcal{P}\backslash \{a_1, a_2,  \cdots, a_{k}, a_{k+1}, b \}$.

 \noindent This completes the proof of Claim~\ref{claim53}.
  \end{proof}

\vspace{3mm}

\begin{claim}\label{claim54}
 $$\{a_{l_1},a_{l_2}, \cdots, a_{l_{k-4}}, x, a_{k}, a_{k+1}, b \}  \not \in \Gamma,$$  where $1 \leq l_1, l_2,\cdots, l_{k-4} \leq k-1$ and $x \in \mathcal{P}\backslash \{a_1, a_2,  \cdots, a_{k}, a_{k+1}, b \}.$
 \end{claim}
 
 \begin{proof}[Proof of Claim~\ref{claim54}]
 
 Let us assume otherwise, that is, $\{a_{l_1},a_{l_2}, \cdots, a_{l_{k-4}}, x, a_k, a_{k+1}, b \}  \in \Gamma$, where $1 \leq l_1, l_2, \cdots, l_{k-4} \leq k-1$ and $x \in \mathcal{P}\backslash \{a_1, a_2, \cdots, a_{k}, a_{k+1}, b \}$.\\
 
 Next, we can consider two participants from $\{a_{1},a_{2}, \cdots, a_{k-1}\}$ who are not in
 $\{a_{l_1},a_{l_2}, \cdots, a_{l_{k-4}}\}$ and denote them as $a_{l_{k-3}}$ and $a_{l_{k-2}}$ .
 \noindent We will now consider  the following subsets of participants in $\mathcal{P}$: $B_1 = \{ b \},\:B_2=\{a_{l_1}, b \}$,  $\:B_3=\{a_{l_1},a_{l_2}, b \},\cdots$, $B_{k-2}=\{a_{l_1},a_{l_2},\cdots,a_{l_{k-4}}, a_{k}, b \},$ 
  $B_{k-1} =$ $\{a_{l_1},a_{l_2},\cdots,a_{l_{k-4}}, a_k, a_{k+1}, b \},$ and $\:B_k=\{a_{l_1},a_{l_2}, \cdots,a_{l_{k-4}},a_{l_{k-3}}, a_k, a_{k+1}, b\}$. 
Note that  $ \emptyset \neq B_1 \subset B_2 \subset \cdots \subset 
\:B_k=\{a_{l_1},a_{l_2}, \cdots,a_{l_{k-3}}, a_k, a_{k+1}, b \} \not \in \Gamma
$. Next, we consider
a subset $A=\{a_{l_2},a_{l_3},  \cdots, a_{l_{k-4}}, a_{l_{k-2}}, a_k , a_{k+1}, x\} \subseteq \mathcal{P}$ consisting of $k-1$ participants. Note that $A \not \in \Gamma$. We now consider the subsets of $A$ as follows: 
 $X_1=\{a_{l_2},a_{l_3},   \cdots, a_{l_{k-4}},a_{l_{k-2}}, a_k, a_{k+1}, x\}$,\ \ $X_2=\{a_{l_2},a_{l_3},   \cdots, a_{l_{k-4}}, a_{l_{k-2}}, a_{k}, x\},$  \ \  $ X_3=$ $
\{a_{l_{3}},a_{l_{4}},  \cdots, a_{l_{k-4}}, a_{l_{k-2}}, a_{k}, x\},$ $\cdots$,
$ X_{k-2}=\{a_{l_{k-2}},x\}$, $\:X_{k-1}=\{x\}$, and $\:X_k=\{a_{l_{k-2}}\}$.  \\

 Based on  the assumption,   we observe  $B_1 \cup X_1 = \{a_{l_2},a_{l_3},  \cdots,a_{l_{k-4}}, a_{l_{k-2}}, x, a_{k}, a_{k+1}, b\}\in\Gamma$.  Additionally, using  Claim~\ref{claim53}, we  obtain that $$\ B_2 \cup X_2 =
\{a_{l_1},a_{l_2}, \cdots, a_{l_{k-4}}, a_{l_{k-2}}, x, a_{k}, b\}  \in\Gamma,$$

\noindent and so on up to $\ B_{k-2} \cup X_{k-2} =\{a_{l_1},a_{l_2},\cdots, a_{l_{k-4}},a_{l_{k-2}}, x, a_{k}, b\}  \in\Gamma$. 
 Based on the assumption, we have 
$\ B_{k-1} \cup X_{k-1} = 
\{a_{l_1},a_{l_2}, \cdots, a_{l_{k-4}}, x, a_k, a_{k+1}, b \}  \in \Gamma$.
Using~\ref{eqn3}, we obtain that $ \{a_{l_1},a_{l_2}, \cdots,a_{l_{k-2}}, a_k, a_{k+1}\}  \in\Gamma$, which implies 
$\ B_{k} \cup X_{k} = \{a_{l_1},a_{l_2}, \cdots,a_{l_{k-2}}, a_k, a_{k+1}, b \}  \in\Gamma$. \\

 Since the set of $k-1$ participants can not be in $\Gamma$, we derive  that $B_0\cup X_1
 \notin \Gamma,$ $ \  B_1 \cup X_2 
\notin \Gamma$, $\cdots,$ $B_{k-2}\cup X_{k-1}  \not \in \Gamma$.
Based on the assumption, we also deduce that
$$B_{k-1} \cup X_{k} = \{a_{l_1},a_{l_2},\cdots,a_{l_{k-4}}, a_{l_{k-2}}, a_k, a_{k+1}, b \}
  \not \in \Gamma.$$  Then the sequence $\phi \neq B_1 \subseteq B_2 \subseteq \cdots \subseteq B_k \notin \Gamma$ is made independent by the set $A=\{a_{l_2},a_{l_3},  \cdots, a_{l_{k-4}}, a_{l_{k-2}}, a_k , a_{k+1}, x\} \subseteq \mathcal{P}$. Hence, using the independent sequence method, we can conclude that $\rho^*(\Gamma) \leq \frac{|A|}{k}=\frac{k-1}{k}$, which leads to  a contradiction. Therefore, we can conclude that 
\begin{align}\label{eq:29}
\{a_{l_1},a_{l_2}, \cdots, a_{l_{k-4}}, x, a_{k}, a_{k+1}, b \}  \not \in \Gamma
\end{align}
 where $1 \leq l_1, l_2,\cdots, l_{k-4} \leq k-1$ and $x \in \mathcal{P}\backslash \{a_1, a_2,  \cdots, a_{k}, a_{k+1}, b \}$. \\

 \noindent This completes the proof of Claim~\ref{claim54}. \\
 \end{proof}

 \begin{claim}\label{claim55} Suppose that the following three statements hold for $1 \leq t \leq k-2$.
 \begin{enumerate}
 \item[(i)]
 $\{a_{l_1},a_{l_2}, \cdots, a_{l_{k-t-1}}, x_1,x_2,\cdots, x_{t-1}, a_{k+1}, b \}  \not \in \Gamma$, 
 \item[(ii)] $\{a_{l_1},a_{l_2}, \cdots, a_{l_{k-t-1}}, x_1,x_2,\cdots, x_{t-1}, a_{k}, b \}  \in \Gamma$,
 \item[(iii)] $\{a_{l_1},a_{l_2}, \cdots, a_{l_{k-t-2}}, x_1,x_2,\cdots, x_{t-1}, a_{k}, a_{k+1}, b \}  \not \in \Gamma$,
 \end{enumerate}
 where $1 \leq l_1, l_2,  \cdots, l_{k-t-1} \leq k-1$ and $x_1, x_2, \cdots, x_{t-1} \in \mathcal{P}\backslash \{a_1, a_2,  \cdots, a_{k}, a_{k+1}, b \}$.\\

 \noindent Then for $1 \leq t \leq k-2$, we have
 $$\{a_{l_1},a_{l_2}, \cdots, a_{l_{k-t-2}}, x_1,x_2,\cdots, x_{t}, a_{k+1}, b \}  \not \in \Gamma,$$ 
 \noindent and 
 $$\{a_{l_1},a_{l_2}, \cdots, a_{l_{k-t-2}}, x_1,x_2,\cdots, x_{t}, a_{k}, b \}   \in \Gamma,$$  where $1 \leq l_1, l_2,  \cdots, l_{k-t-2} \leq k-1$ and $x_1, x_2, \cdots, x_{t} $ $ \in \mathcal{P}\backslash \{a_1, a_2,  \cdots, a_{k}, a_{k+1}, b \}$. \end{claim}

 \begin{proof}[Proof of Claim~\ref{claim55}]
Let us consider  $k+2$ different participants in $\mathcal{P}$, denoted as
$a_{l_1}, \cdots, a_{l_{k-t-1}},$ $ x_1,\cdots,x_t, a_k, a_{k+1}, b$,  where $1 \leq l_1,  \cdots, l_{k-t-1} \leq k-1$ and $x_1,  \cdots, x_t $ $ \in \mathcal{P}\backslash \{a_1, a_2,  \cdots, a_{k}, a_{k+1}, b \}$. \\

\noindent Using the properties $(i), (ii), (iii)$, now we apply Lemma~\ref{theorem1} with $u_i=a_{l_{k-t-1}}, u_j = a_k, v= x_t$.  Then either

 $$ w(\{a_{l_1}, \cdots, a_{l_{k-t-1}}, x_1,\cdots, x_{t-1}, a_k, a_{k+1}, b \},  \ \Gamma ) = k+1 $$ or
 $$ \{ a_{l_1}, \cdots, a_{l_{k-t-2}}, x_1,\cdots, x_{t-1},x_t,  a_{k+1}, b \} \not \in \Gamma ,$$
 
  \noindent where $1 \leq l_1,\cdots, l_{k-t-1} \leq k-1$ and $x_1,  \cdots, x_t \in \mathcal{P}\backslash \{a_1,   \cdots, a_{k}, a_{k+1}, b \}$. \\

 From $(i)$ and $(iii)$, we observe that  $w(\{a_{l_1}, \cdots, a_{l_{k-t-1}}, x_1,\cdots, x_{t-1}, a_k, a_{k+1}, b \},  \ \Gamma ) \neq k+1$. Therefore, we can conclude that
 
 $$ \{ a_{l_1}, \cdots, a_{l_{k-t-2}}, x_1,\cdots, x_t,  a_{k+1}, b \} \not \in \Gamma ,$$
 
  \noindent where $1 \leq l_1,  \cdots, l_{k-t-1} \leq k-1$ and $x_1, \cdots,  x_t \in \mathcal{P}\backslash \{a_1,  \cdots, a_{k}, a_{k+1}, b \}$. \\
  
   \noindent Since  $a_k$ is not equivalent to $a_{k+1}$, for $1 \leq t \leq k-2$, we also conclude that
$$  \{a_{l_1}, \cdots, a_{l_{k-t-2}}, x_1,\cdots, x_t, a_{k}, b \} \in \Gamma,$$
\noindent where $1 \leq l_1,  \cdots, l_{k-t-2} \leq k-1$ and $x_1,  \cdots, x_{t} \in \mathcal{P}\backslash \{a_1,   \cdots, a_{k}, a_{k+1}, b \}$. \\ This completes the proof of Claim~\ref{claim55}. \\
 \end{proof}

\noindent Using Claim~\ref{claim55}, we can conclude that
   $$\{a_{k+1}, b\} \not \subset A {{\text \ \  \  {for \  any} \  \ }}A \in \Gamma_0,$$

\noindent where $\Gamma_0$ is the collection of minimal qualified subsets.\\

 Now we are prepared to demonstrate the property $(2)$ by using the property $(1)$. To accomplish this, we will begin by
 establishing the following claims.\\

\begin{claim}\label{claim511}

Let $A$ be a subset in $ \mathcal{P}\backslash \{ a_{k+1}, b \}$. If $A \cup \{a_{k+1}\}\in \Gamma$, then  $A \cup \{ b \} \in \Gamma$.
\end{claim}

\begin{proof}[Proof of Claim~\ref{claim511}] Let us consider $k+2$ distinct participants in $ \mathcal{P}$, denoted as $a_{p_1},a_{p_2},\cdots, a_{p_{k-1}}, $ $ a_k, a_{k+1}, b$, where
$ \{ a_{p_1}, \cdots, a_{p_{k-1}}\} \in \mathcal{P}\backslash \{ a_k, a_{k+1}, b \}$.
First, let $A=\{a_{p_1},,\cdots, a_{p_{k-2}}, a_k \}$ be a subset in $ \mathcal{P}\backslash \{ a_{k+1},   b \}$.
Suppose that $A \cup \{a_{k+1}\} = \{a_{p_1},\cdots, a_{p_{k-2}}, a_k, a_{k+1} \} \in \Gamma$. Using the property $(1)$, we clearly deduce that $A \cup \{b\} = \{a_{p_1},\cdots, a_{p_{k-2}}, a_k, b \} \in \Gamma$, where
$ a_{p_1}, \cdots, a_{p_{k-2}}$ is in $ \mathcal{P}\backslash \{ a_k, a_{k+1}, b \}$.
Next, let $A=\{a_{p_1},\cdots, a_{p_{k-1}} \}$ be a subset in $ \mathcal{P}\backslash \{ a_k, a_{k+1},   b \}$. Suppose that $A \cup \{a_{k+1}\} = \{a_{p_1},\cdots, a_{p_{k-1}}, a_{k+1} \} \in \Gamma$, where $ a_{p_1}, \cdots, a_{p_{k-1}}$ is in $ \mathcal{P}\backslash \{ a_k, a_{k+1}, b \}$.  Assuming the contrary, i.e., 
$A \cup \{b\} = \{a_{p_1},\cdots, a_{p_{k-1}}, b \} \not \in \Gamma$, where $ a_{p_1}, \cdots, a_{p_{k-1}}$ are in $ \mathcal{P}\backslash \{a_k,  a_{k+1}, b \}$, we can proceed to prove the desired result. \\

 Let us first define the following subsets of participants  in $\mathcal{P}$: $B_1 = \{b\},\:B_2=\{ a_{p_2}, b \}$,  $\:B_3=\{a_{p_2}, a_{p_3}, b \},\cdots, B_{k-2}=\{a_{p_2},\cdots,a_{p_{k-3}}, a_{p_{k-2}}, b \},\:B_{k-1}=\{a_{p_1},a_{p_2},\cdots,a_{p_{k-2}}, b \},$ and $\:B_k=\{a_{p_1},a_{p_2},\cdots,a_{p_{k-2}}, a_{k+1}, b\}$. 
Based on the property $(1)$, we observe that  $ \emptyset \neq B_1 \subset B_2 \subset \cdots \subset 
\:B_k= \{a_{p_1}, \cdots,a_{p_{k-2}}, a_{k+1},  b \}
\not \in \Gamma$. Next, we  consider
a subset $S=\{a_{p_2},  \cdots, a_{p_{k-1}}, a_{k} \} \subseteq \mathcal{P}$ consisting of $k-1$ participants. Note that $S \not \in \Gamma$. We now define the following subsets of $S$:  $X_1=\{a_{p_2}, a_{p_3}, \cdots, a_{p_{k-1}}, a_{k}
\},$ $ \:X_2=\{a_{p_3}, \cdots, a_{p_{k-1}}, a_{k}\},\:X_3=
\{a_{p_{4}}, \cdots, a_{p_{k-1}}, a_{k}\}$
$,\cdots,X_{k-3}= \{a_{p_{k-2}},a_{p_{k-1}}, a_{k}\},$ $ X_{k-2}=\{a_{p_{k-1}},a_{k}\}$, $\:X_{k-1}=\{a_{k}\}$, and $\:X_k=\{a_{p_{k-1}}\}$.  \\

Using the property $(1)$,   we observe  that $B_1 \cup X_1 = \{a_{p_2},a_{p_3}, \cdots,a_{p_{k-1}}, a_k, b \}\in\Gamma$, and so on up to 
$ B_{k-2} \cup X_{k-2} =  \{a_{p_2},a_{p_3}, \cdots,a_{p_{k-1}}, a_k, b\} \in \Gamma$. Additionally, we obtain that
$ B_{k-1} \cup X_{k-1} =  \{a_{p_1},a_{p_2}, \cdots,a_{p_{k-2}}, a_k, b\} \in \Gamma$.
Based on the assumption, we also have  $ \{a_{p_1},\cdots, a_{p_{k-1}}, a_{k+1} \}  $ $\in \Gamma$, which implies that
$ B_{k}\cup X_{k} = \{a_{p_1}, \cdots,a_{p_{k-1}},  a_{k+1}, b \}  \in\Gamma$. Since the set of $k-1$ participants can not be in $\Gamma$, we derive that $B_0\cup X_1 \not \in \Gamma$, $\cdots$, $B_{k-2}\cup X_{k-1} \not \in \Gamma$. Based on the assumption, we also have $B_{k-1}\cup X_k = \{a_{p_1},a_{p_2},\cdots, a_{p_{k-1}}, b \} \not \in \Gamma$. \\

Therefore, we can apply the independent sequence method to the sequence $\phi \neq B_1 \subseteq B_2 \subseteq \cdots \subseteq B_k \notin \Gamma$ with the set $S= \{a_{p_2},a_{p_3},  \cdots, a_{p_{k-1}}, a_{k} \} $, to obtain  $\rho^*(\Gamma) \leq \frac{|S|}{k}=\frac{k-1}{k}$, which leads to a contradiction. Thus, we can conclude that $A \cup \{b\} = \{a_{p_1},a_{p_2},\cdots, a_{p_{k-1}}, b \}  \in \Gamma$, where $ a_{p_1}, \cdots, a_{p_{k-1}}$ is in $ \mathcal{P}\backslash \{ a_k, a_{k+1}, b \}$. This completes the proof of Claim~\ref{claim511}.
\end{proof}

\begin{claim}\label{claim522}
Let $A$ be a subset in $ \mathcal{P}\backslash \{ a_{k+1}, b \}$. If $A \cup \{b \}\in \Gamma$, then we have $A \cup \{ a_{k+1} \} \in \Gamma$.
\end{claim}

\begin{proof}[Proof of Claim~\ref{claim522}]
 Let us consider $k+2$ distinct participants in $\mathcal{P}$, denoted as $a_{p_1},a_{p_2},\cdots, a_{p_{k-1}},  $ $a_k, a_{k+1}, b$, where
$ a_{p_1}, \cdots, a_{p_{k-1}} \in \mathcal{P}\backslash \{ a_k, a_{k+1}, b \}$. First, let $A=\{a_{p_1},\cdots, a_{p_{k-2}}, a_k \}$ be a subset in $ \mathcal{P}\backslash \{ a_{k+1},   b \}$.
Suppose that $A \cup \{b\} = \{a_{p_1},\cdots, a_{p_{k-2}}, a_k, b \} \in \Gamma$, where
$ a_{p_1}, \cdots, a_{p_{k-2}}$ are in $ \mathcal{P}\backslash \{ a_k, a_{k+1}, b \}$. Assuming the contrary, i.e., 
$A \cup \{a_{k+1}\} = \{a_{p_1},\cdots, a_{p_{k-2}}, a_k, a_{k+1} \} \not \in \Gamma$, where $ a_{p_1}, \cdots, a_{p_{k-2}}$ are in $ \mathcal{P}\backslash \{ a_k, a_{k+1}, b \}$. Using the property $(1)$, we obtain $\{a_{p_1},\cdots, a_{p_{k-2}}, a_{k+1}, b \} $ $\not \in \Gamma$ and $\{a_{s_1},\cdots, a_{s_{k-3}}, a_k, a_{k+1}, b \} $ $\not \in \Gamma$, where $s_1, s_2, \cdots s_{k-3} \in \{p_1, p_2, \cdots, p_{k-2}\}$ and 
$ a_{p_1}, \cdots, a_{p_{k-2}}$ are in $ \mathcal{P}\backslash \{ a_k, a_{k+1}, b \}$. Then, this contradicts the fact that  $1 \not \in \Omega(k+1, \Gamma)$. Therefore, we can conclude that $A \cup \{a_{k+1}\} = \{a_{p_1},\cdots, a_{p_{k-2}}, a_k, a_{k+1} \}  \in \Gamma$, where $ a_{p_1}, \cdots, a_{p_{k-2}}$ are in $ \mathcal{P}\backslash \{ a_k, a_{k+1}, b \}$. \\

 Next, let $A=\{a_{p_1},a_{p_2},\cdots, a_{p_{k-1}} \}$ be a subset in $ \mathcal{P}\backslash \{ a_k, a_{k+1}, b \}$. Suppose that $A \cup \{b\} = \{a_{p_1},\cdots, a_{p_{k-1}}, b \} \in \Gamma$, where $ a_{p_1}, \cdots, a_{p_{k-1}}$ are in $ \mathcal{P}\backslash \{ a_k, a_{k+1}, b \}$.  Assuming the contrary, i.e., 
$A \cup \{a_{k+1}\} = \{a_{p_1},\cdots, a_{p_{k-1}}, a_{k+1} \} \not \in \Gamma$, where $ a_{p_1}, \cdots, a_{p_{k-1}}$ are in $ \mathcal{P}\backslash \{ a_k, a_{k+1}, b \}$. Moreover, using the property $(1)$, we derive that\begin{align}\label{eq:21}
 \{ a_{\beta_1},  \cdots a_{\beta_{k-2}},  a_{k+1}, b \} \not \in \Gamma,
 \end{align}

 \noindent where, $ a_{\beta_1}, \cdots a_{\beta_{k-2}} \in \{ a_{p_1}, a_{p_2}, \cdots a_{p_{k-1}}\}.$\\

 Then, let us consider $k+1$ participants $a_{p_1},a_{p_2}, \cdots, a_{p_{k-1}}, a_{k+1}, b$, then the above results lead to a contradiction, since the number of minimal qualified subsets contained in any set of $k+1$ participants is not equal to  $1$. Thus, we can conclude that $A \cup \{a_{k+1}\} = \{a_{p_1},\cdots, a_{p_{k-1}}, a_{k+1} \} \in \Gamma$, where $ a_{p_1},  \cdots, a_{p_{k-1}}$ are in $ \mathcal{P}\backslash \{ a_k, a_{k+1}, b \}$. This completes the proof of Claim~\ref{claim522}.
\end{proof}

\noindent Using Claim~\ref{claim511} and~\ref{claim522}, we can conclude that \\

if $A \subset \mathcal{P}\backslash \{ a_{k+1}, b \}$, then $A \cup \{a_{k+1}\}\in \Gamma$ if and only if $A \cup \{ b \} \in \Gamma$,\\
 
\noindent where $\Gamma_0$ is the collection of minimal qualified subsets.\\

\noindent Using the property $(1)$ and $(2)$, we  conclude that $b$ is equivalent to  $a_{k+1}$. \\

\noindent {\bf{ Case II: $ \{ a_{\Delta_1}, a_{\Delta_2}, \cdots a_{\Delta_{k-2}}, a_{k}, b \} \not \in \Gamma$}}, where $1 \leq \Delta_1, \Delta_2, \cdots \Delta_{k-2} \leq k-1$.\\

\noindent In this case, we  aim to demonstrate the equivalence of $a_{k}$  and $b$. Given that  $a_{k}$ and $b$ represent distinct participants, it is necessary to establish that:

\begin{itemize}
\item[(1)] $\{a_{k}, b\} \not \subset A$ for any $A \in \Gamma_0$,
\item[(2)] if $A \subset \mathcal{P}\backslash \{ a_{k}, b \}$, then $A \cup \{a_{k}\}\in \Gamma$ if and only if $A \cup \{ b \} \in \Gamma$,
\end{itemize}

\noindent where $\Gamma_0$ is the collection of minimal qualified subsets.\\

\noindent The proof of Case II is omitted since it is similar to the proof of Case I by replacing $a_{k}$ with $a_{k+1}$.\\
 \end{proof}


\section {Proof of Theorem~\ref{main:mainthm}}

\noindent Now, we are ready to establish Theorem~\ref{main:mainthm}. In this section, we we will prove Theorem~\ref{main:mainthm} by utilizing Lemma \ref{lemma2}, \ref{lemma3}, and \ref{lemma4}.

\begin{proof}[Proof of Theorem~\ref{main:mainthm}]
$ $ \newline
(1) $\Rightarrow$ (2) :
Since the vector space access structures are ideal using Theorem~\ref{idealtheorem}, the proof is complete. \\
(2) $\Rightarrow$ (3) : As the optimal information rate of an ideal access structure equals one, the proof is complete. \\
(3) $\Rightarrow$ (4) : Let $\Gamma_\sim$ be a reduced access structure of $\Gamma$ on a set of participants $\mathcal{P}$. Since $k+1 \in \Omega(k+1,\Gamma_\sim)$, there exist $k+1$ distinct participants $a_1, a_2, \cdots , a_k, a_{k+1} \in \mathcal{P}$, such that $\omega(\{a_1,a_2, \cdots, a_k, a_{k+1}\},\Gamma_\sim)=k+1$. Consequently, we can conclude that the induced access structure $\Gamma_\sim(\{a_1,a_2,  \cdots,a_k, a_{k+1}\})$ is a $(k,k+1)$-threshold access structure.\\

Let us consider $\mathcal{P}'\subset \mathcal{P}$, consisting of $m$ participants, where $m\geq k+1$ and $\{a_1,a_2,\dots,a_{k+1}\}\subset\mathcal{P}'$, such that the induced structure $\Gamma_\sim(\mathcal{P}')$ is the $(k,m)$-threshold access structure, while $\Gamma(\mathcal{P}'\cup \{b\})$ is not $(k, m+1)$-threshold access structure for $b\in\mathcal{P}\setminus\mathcal{P}'$. We now assert that $\mathcal{P}=\mathcal{P}'$. To demonstrate this, let us assume the contrary, that is,  $\mathcal{P}\neq\mathcal{P}'$. Then, we claim that there must exist $a_1,a_2,\dots,a_{k+1}\in\mathcal{P}'$ and $b\in\mathcal{P}\setminus\mathcal{P}'$ such that 
\begin{itemize}
    \item[(i)] $\{a_1,a_2,\dots,a_{k-1},b\}\in\Gamma$
    \item[(ii)] $\{a_{t_1},a_{t_2},\dots,a_{t_{k-3}},a_k,a_{k+1},b\}\nin\Gamma$, where $1\leq t_1,t_2,\dots,t_{k-3}\leq k-1$
    
    or $\{a_{l_1},a_{l_2},\dots,a_{l_{k-2}},a_k,b\}\nin\Gamma$, where $1\leq l_1,l_2,\dots,l_{k-2}\leq k-1$
    
    or $\{a_{s_1},a_{s_2},\dots,a_{s_{k-2}},a_{k+1},b\}\nin\Gamma$, where $1\leq s_1,s_2,\dots,s_{k-2}\leq k-1$.\\
\end{itemize}

To begin, let us establish  property $(i)$. Given that $\mathcal{P}\neq\mathcal{P}'$, there exists $A\in {\Gamma_\sim}_0$ such that $A \cap \mathcal{P}' \neq \phi$, and $A\not \subset \mathcal{P'}$. Now, let us consider $k+3$ participants, denoted as $a_1,a_2,\dots,a_{k+1},b_1,b_2 \in \mathcal{P}$, where $a_1,a_2,\dots,a_{k+1}\in \mathcal{P}'$, and $b_1,b_2 \in \mathcal{P}\setminus\mathcal{P}'$. We have two cases to consider. First, if $|A\cap\mathcal{P}'|=k-1$, then  we can conclude that $A=\{a_1,a_2,\dots,a_{k-1},b_i\}\in\Gamma$ for some $a_1,a_2,\dots,a_{k-1}\in\mathcal{P}'$ and $b_i\in\mathcal{P}\setminus\mathcal{P}'$, where $i=1,2$. Second, if $|A\cap\mathcal{P}'|\neq k-1$, then we deduce that $A=\{x_1,x_2,\dots,x_{k-2},b_1,b_2\}\in\Gamma$, where $x_1,x_2,\dots,x_{k-2}\in\mathcal{P}'$ and $b_1,b_2\in \mathcal{P}\setminus\mathcal{P}'$.\\

Let us consider three distinct participants, denoted as $x_{k-1},x_k,x_{k+1}\in \mathcal{P}' \setminus \{x_1,x_2,\cdots,x_{k-2}\}$. Since $x_1,x_2,\cdots,x_{k-1},x_k, x_{k+1}\in\mathcal{P}'$ and $\Gamma_\sim(\mathcal{P}')$ is a $(k,m)$-threshold access structure, we have  $\omega(\{x_1,x_2,\cdots,x_k,x_{k+1}\},\Gamma_\sim)=k+1$, which satisfies the condition of Lemma \ref{lemma2}. Therefore, by utilizing Lemma \ref{lemma2}, there exist $j_1,j_2,\dots,j_{k-1}$ such that either $$\{x_{j_1},x_{j_2},\dots,x_{j_{k-1}},b_1\}\in\Gamma$$ or $$\{x_{j_1},x_{j_2},\dots,x_{j_{k-1}},b_2\}\in\Gamma.$$ Now, property $(i)$ holds for the second case with $x_{j_1}=a_1, x_{j_2}=a_2, \cdots, x_{j_{k-1}}=a_{k-1}$, and $b_1=b$.  \\

Now, we establish property $(ii)$. Since $\Gamma_\sim(\mathcal{P}')$ is a $(k,m)$-threshold access structure and $\Gamma_\sim(\mathcal{P}'\cup\{b\})$ is not a $(k,m+1)$-threshold access structure, where $b \nin \mathcal{P}'$, there must exist two participants, denoted as $y_k,y_{k+1}\in\mathcal{P}'$, such that
at least one of the following conditions holds:

    \begin{align}\label{tt1}\{a_{t_1},a_{t_2},\dots,a_{t_{k-3}},y_k,y_{k+1},b\}\nin\Gamma,\: \text{where}\:\: 1\leq t_1,t_2,\dots,t_{k-3}\leq k-1 \end{align}
    \begin{align}\label{tt2}\text{or}\:\:\:\{a_{l_1},a_{l_2},\dots,a_{l_{k-2}},y_{k},b\}\nin\Gamma,\: \text{where}\:\: 1\leq l_1,l_2,\dots,l_{k-2}\leq k-1\end{align}
    \begin{align}\label{tt3}\text{or}\:\:\:\{a_{s_1},a_{s_2},\dots,a_{s_{k-2}},y_{k+1},b\}\nin\Gamma,\: \text{where}\:\: 1\leq s_1,s_2,\dots,s_{k-2}\leq k-1,
    \end{align}
    where $b\in\mathcal{P}\setminus\mathcal{P}'$. \\

    \noindent  From property $(i)$, we can conclude that there exist $a_1,a_2,\dots,a_{k+1}\in\mathcal{P}'$ and $b\in\mathcal{P}\setminus\mathcal{P}'$ such that 
 $\{a_1,a_2,\dots,a_{k-1},b\}\in\Gamma$. We must now consider two cases. \\
 
First, if $\{a_1,a_2,\dots,a_{k-1}\} \cap \{y_k,y_{k+1}\}=\phi$, then property $(ii)$ holds  with $y_k=a_k$ and $y_{k+1}=a_{k+1}$.
    Second, without loss of generality, we assume that $a_1,a_2,\dots,a_{k-1},y_k$ are distinct, and $y_{k+1}=a_1$. Since $|\mathcal{P}'|=m'\geq k+1$, there exists $z_{k+1}$ such that $z_{k+1}\in \mathcal{P}'\setminus \{a_1,a_2,\dots, a_{k-1},y_k\}$. Consequently,  we can deduce that $\Gamma_\sim(\{a_1,a_2,\dots,a_{k-1},y_k,z_{k+1},b\})$ is not a $(k,k+2)$-threshold access structure by utilizing Equations ($~\ref{tt1}),(\ref{tt2}),(\ref{tt3}).$
  Now, we assert that
    $$\{a_{t_1},a_{t_2},\dots,a_{t_{k-3}},y_k,z_{k+1},b\}\nin\Gamma,\: \text{where}\:\: 1\leq t_1,t_2,\dots,t_{k-3}\leq k-1$$
    $$\text{or}\:\:\:\{a_{l_1},a_{l_2},\dots,a_{l_{k-2}},y_k,b\}\nin\Gamma,\: \text{where}\:\: 1\leq l_1,l_2,\dots,l_{k-2}\leq k-1$$
    $$\text{or}\:\:\:\{a_{s_1},a_{s_2},\dots,a_{s_{k-2}},z_{k+1},b\}\nin\Gamma,\: \text{where}\:\: 1\leq s_1,s_2,\dots,s_{k-2}\leq k-1.$$\\
    
\noindent To accomplish this, let us assume otherwise, that is,  suppose that
    $$\{a_{t_1},a_{t_2},\dots,a_{t_{k-3}},y_k,z_{k+1},b\}\in\Gamma,\: \text{where}\:\: 1\leq t_1,t_2,\dots,t_{k-3}\leq k-1$$
    $$\text{and}\:\:\:\{a_{l_1},a_{l_2},\dots,a_{l_{k-2}},y_k,b\}\in\Gamma,\: \text{where}\:\: 1\leq l_1,l_2,\dots,l_{k-2}\leq k-1$$
    $$\text{and}\:\:\:\{a_{s_1},a_{s_2},\dots,a_{s_{k-2}},z_{k+1},b\}\in\Gamma,\: \text{where}\:\: 1\leq s_1,s_2,\dots,s_{k-2}\leq k-1.$$ \\

\noindent Using property $(i)$, we can find  $a_1,a_2,\dots,a_{k-1}$ such that $\{a_1,a_2,\dots,a_{k-1},b\}\in\Gamma$, as required by Lemma \ref{lemma3}. Consequently, $\Gamma_\sim(\{a_1,a_2,\dots,a_{k-1},y_k,z_{k+1},b\})$ is a $(k,k+2)$-threshold access structure, which creates a contradiction. Therefore, we can conclude that property $(ii)$ holds with $y_k=a_k$ and $z_{k+1}=a_{k+1} 
$. Since $a_1,a_2, \cdots, a_{k+1}\in \mathcal{P}'$ and
$\Gamma_\sim(\mathcal{P}')$ is $(k,m)$-threshold access structure, we have $\omega(\{a_1,a_2,\dots,a_{k+1}\},\Gamma_\sim)=k+1$.  Utilizing Lemma \ref{lemma4},  we can conclude that either $b$ is equivalent to $a_k$ or $b$ is equivalent to $a_{k+1}$. This contradicts  the fact that $\Gamma_\sim$ is a reduced access structure. Therefore, we can deduce that $\mathcal{P}=\mathcal{P}'$. Consequently, $\Gamma_\sim(\mathcal{P})$ is also $(t,m)$-threshold access structure. \\

\noindent (4) $\Rightarrow$ (1) : The results from Section $3$ establish that (4) implies (1). According to the definitions, a $(k,n)$-threshold access structure is an example of a vector space access structure. Therefore, if the reduced access structure of $\Gamma$ is a $(k,n)$-threshold access structure, then it is also  a vector space access structure. Furthermore, the reduced access structure of $\Gamma$ is a vector space access structure if and only if $\Gamma$ is a vector space access structure. \\

\noindent This completes the proof of Theorem~\ref{main:mainthm}.
\end{proof}
 




\end{document}